\documentclass[11pt,leqno]{article}

\usepackage{amsmath,amsthm,amssymb}
\usepackage{graphicx}
\usepackage{float}
\usepackage{mathabx}
\usepackage{url}
\usepackage{array}
\usepackage{authblk}
\usepackage{geometry}
\theoremstyle{plain}
\newtheorem{thm}{Theorem}
\newtheorem{prop}{Proposition}
\newtheorem{cor}{Corollary}

\theoremstyle{definition}
\newtheorem{ex}{Example}

\newtheorem{assum}{Assumption}

\theoremstyle{remark}
\newtheorem{rem}{Remark}

\DeclareMathOperator{\tr}{tr}
\title{Exponential-Integrator Error Bounds for Reverse-Time Diffusions\thanks{Earlier
versions of this work appeared under the title ``Structural and Convergence
Analysis of Discrete-Time Denoising Diffusion Probabilistic Models.'' The
present version substantially revises and extends the analysis.}}
\author[1]{Yumiharu Nakano}
\affil[1]{Department of Mathematical and Computing Science\\
Institute of Science Tokyo\\
\texttt{nakano@comp.isct.ac.jp}}
\date{\today}

\begin{document}

\maketitle

\begin{abstract}
We develop a total-variation error analysis for exponential-integrator sampling beyond the Gaussian Ornstein--Uhlenbeck forward processes used in standard diffusion-model analyses.  The forward drift is a nonautonomous linear field plus a bounded Lipschitz perturbation.  The diffusion coefficient is time dependent, and the grid may be nonuniform.  This class includes variance-preserving and variance-exploding models and verified nonlinear examples.  Under low regularity, the time-discretization error has a square-root rate in the maximum time step.  For exact scores, smooth-density regularity improves the rate to first order.  The end-to-end bound separates initialization, learned-score, and time-discretization errors.  The proof combines a uniform Fisher-information bound, reverse-score energy, and a one-step density expansion.  The low-regularity result holds without a pointwise score-Hessian bound.
\end{abstract}

\noindent\textbf{Keywords.}
Reverse-time stochastic differential equations, exponential integrators,
total-variation error analysis, nonuniform time discretization,
score-based generative models.

\medskip
\noindent\textbf{2020 Mathematics Subject Classification.}
65C30, 60H35, 60J60.

\section{Introduction}\label{sec:intro}

\subsection{Background}

Score-based models~\cite{sohl-etal:2015,ho-etal:2020,son-etal:2020} reverse a
diffusion toward $\mu_{data}$.  Our exponential integrator, related
to~\cite{kom-bur:2014}, treats the linear and Gaussian parts of the reverse
stochastic differential equation (SDE) exactly and freezes the score.

Most analyses exploit the Ornstein--Uhlenbeck (OU) forward process
\cite{chen-etal:2023,ben-etal:2023,lee-etal:2022,lee-etal:2023,
con-etal:2025,li-yan:2024,jain-zhang:2025}, while linear schedules have been
treated in Wasserstein distance under log-concavity~\cite{gao-etal:2025}.  We
allow the forward diffusion and schedule to vary, including bounded nonlinear
perturbations of Gaussian references (Section~\ref{sec:doob_perturbation}).  We
seek bounds beyond this Gaussian structure rather than optimized OU rates.

\subsection{The class of forward dynamics}

We consider
\begin{equation*}
 dX_t = \big(A(t)X_t + b(t,X_t)\big)\,dt + \sigma(t)\,dW_t,
 \qquad X_0\sim\mu_{data},
\end{equation*}
where $W$ is an $m$-dimensional Brownian motion, $A(t)$ is a
(non-autonomous) matrix, $b$ is a \emph{bounded, Lipschitz} vector field, and
the diffusion coefficient $\sigma(t)$ depends only on time
with $a(t):=\sigma(t)\sigma(t)^{\mathsf{T}}$ uniformly elliptic.
This \emph{linear-plus-bounded} class contains the denoising diffusion
probabilistic model (DDPM) with an arbitrary finite noise schedule, the
variance-preserving (VP) and variance-exploding (VE) SDEs
\cite{ho-etal:2020,son-etal:2020}, and bounded perturbations of OU adapted to
multimodal data.  It also preserves the structure needed for analysis: the
linear part is integrated exactly, while kernel comparison and an
integration-by-parts representation control the score.  For nonlinear
perturbations the required endpoint derivative estimate is stated explicitly
in Assumption~\ref{lem:kernel}; it is automatic in the linear VP and VE
specializations.

\subsection{Our contribution}

\noindent\textbf{(a) Error bounds for the discrete-time sampler.}
For a time grid with maximum step size $\Delta_{\max}$, the reverse-score
energy estimate, contraction in total variation (TV), and initialization
measured by Kullback--Leibler (KL) divergence give the end-to-end bound
(Theorem~\ref{thm:disc_score} and
Corollary~\ref{cor:main}), schematically
\[
 D_{TV}(\mu_{data},\hat{\mu})
 \lesssim \sqrt{D_{KL}(\nu\,\|\,\mu_T)
 +\Lambda_0T L_{\mathrm{rev}}
 +\Delta_{\max}\Lambda_0[J_1+(1+M_A+L_b)T\bar J+TG_b^2]}\,.
\]
The three terms are initialization, reverse-grid score error, and
discretization; no pointwise score-Hessian assumption is used.  The
low-regularity rate is $O(\Delta_{\max}^{1/2})$, while
Theorem~\ref{thm:first_order_tv} gives $O(\Delta_{\max})$ under
smooth-density regularity.  The VP and VE corollaries provide explicit
Gaussian-initialization bounds, and the componentwise experiment tests all
three error sources.  The entropy direction does not by itself imply a
Wasserstein bound (Remark~\ref{cor:wass}).

\noindent\textbf{(b) Analytical mechanism.}
Under Assumption~\ref{lem:kernel}, Theorem~\ref{thm:fisher_bound} proves the
uniform bound $\sup_{t\in(0,T]}\mathbb{E}|\nabla\log p_t(X_t)|^2\le\bar J$.
Together with reverse-score energy, it yields the low-regularity result
without a pointwise score-Hessian bound.  A Bally--Talay-type density
expansion then gives first order for the exact-score integrator.  The general
nonlinear theorem is conditional on the kernel hypothesis, while
Proposition~\ref{prop:noncancelling} verifies these conditions and the kernel
estimates for a stationary nonlinear class with a nonzero reverse drift.

\noindent\textbf{Positioning.}
Table~\ref{tab:comparison} compares coefficient classes, metrics, and rates.
The relevant axis is the admissible forward dynamics rather than the optimal
rate within one OU model.

\begin{table}[tb]
\centering
\caption{Representative analytical regimes. Individual variants may differ.}
\label{tab:comparison}
\footnotesize
\setlength{\tabcolsep}{3pt}
\renewcommand{\arraystretch}{1.02}
\begin{tabular}{@{}>{\raggedright\arraybackslash}p{.17\linewidth}>{\raggedright\arraybackslash}p{.24\linewidth}>{\raggedright\arraybackslash}p{.27\linewidth}>{\raggedright\arraybackslash}p{.24\linewidth}@{}}
\hline
Work/class & Dynamics and grid & Metric and assumptions & Conclusion \\
\hline
OU/DDPM analyses
\cite{chen-etal:2023,ben-etal:2023} & OU/VP, prescribed grids
& TV or KL, with $L^2$ score error and score regularity or early stopping
& nonasymptotic model-specific bounds with learned scores \\
Gao--Nguyen--Zhu~\cite{gao-etal:2025} & scalar linear schedules, uniform grid
& $W_2$, with accurate and time-regular scores and smooth strongly log-concave data
& iteration bounds for VP, VE, and other schedules \\
Bally--Talay~\cite{bal-tal:1996,bal-tal:1996b}
& smooth uniformly elliptic SDE, uniform-step Euler
& bounded-test weak and pointwise density errors, with bounded smooth coefficients
& first-order weak and density expansions \\
Present work & nonautonomous matrix-linear-plus-bounded, nonuniform grids
& TV, with a kernel hypothesis and no pointwise score Hessian at low regularity
& half order with learned scores and first order with exact scores when smooth \\
\hline
\end{tabular}
\end{table}

\noindent\textbf{Notation.}
$\nabla$ is the spatial gradient; $\nabla^2$, the Hessian.
For a vector field $f$, $(\nabla f)_{ij}=\partial_{x_j}f_i$ is the Jacobian
and $\|\cdot\|_{\mathrm{op}}$ the operator norm.
For a symmetric positive semidefinite matrix $a$ and $u\in\mathbb{R}^d$,
$|u|_a^2:=u^{\mathsf{T}}a\,u$.
$\tr$ is the trace; $A^{\mathsf{T}}$, the transpose;
$M^{-\mathsf{T}}:=(M^{-1})^{\mathsf{T}}$ for an invertible matrix $M$;
$A_{\mathrm{sym}}:=\tfrac{1}{2}(A+A^{\mathsf{T}})$.
For positive definite $V$, set
$\varphi_V(z):=(2\pi)^{-d/2}(\det V)^{-1/2}
\exp(-z^{\mathsf T}V^{-1}z/2)$, the centered Gaussian density with covariance~$V$.
$\mathcal{P}(\mathcal{X})$ is the set of Borel probability measures on a
Polish space~$\mathcal{X}$;
$D_{TV}$, $D_{KL}$, and $W_2$ are the total variation distance, the
Kullback--Leibler divergence, and the quadratic Wasserstein distance;
$F_\#\mu$ is the pushforward of $\mu$ by $F$.

\section{Setting and structural identities}\label{sec:setup}

Fix $T>0$.
Let $(\Omega,\mathcal{F},\mathbb{P})$ be a complete probability space carrying
a filtration $\mathbb{F}=\{\mathcal{F}_t\}_{0\le t\le T}$ satisfying the usual
conditions and an $m$-dimensional $\mathbb{F}$-Brownian motion
$\{W_t\}_{0\le t\le T}$.
Let $\mu_{data}\in\mathcal{P}(\mathbb{R}^d)$ have density~$p_{data}$,
and let $\mathbf{x}_0\sim\mu_{data}$ be $\mathcal{F}_0$-measurable.
The forward dynamics is
\begin{equation}
\label{eq:forward}
 dX_t = \big(A(t)X_t + b(t,X_t)\big)\,dt + \sigma(t)\,dW_t,
 \qquad X_0=\mathbf{x}_0, \qquad t\in[0,T].
\end{equation}
We use the following conditions:

\noindent\textbf{Data condition.}
\begin{enumerate}
\item[(D)] $p_{data}$ admits a weak gradient and
\[
 M_2:=\mathbb E|\mathbf x_0|^2<\infty,
 \qquad
 J_1:=\int_{\{p_{data}>0\}}
 \frac{|\nabla p_{data}(x)|^2}{p_{data}(x)}\,dx<\infty.
\]
\end{enumerate}
We set $\nabla\log p_{data}=0$ on $\{p_{data}=0\}$, so that
$J_1=\mathbb E|\nabla\log p_{data}(\mathbf x_0)|^2$.

\noindent\textbf{Coefficient condition.}
\begin{enumerate}
\item[(F)] $A:[0,T]\to\mathbb{R}^{d\times d}$ is measurable and essentially
bounded, with $\|A\|_{L^\infty}=:M_A<\infty$;
$b:[0,T]\times\mathbb{R}^d\to\mathbb{R}^d$ satisfies
$\sup_{t,x}|b(t,x)| =: R_0<\infty$, $b(t,\cdot)\in C^1(\mathbb{R}^d)$ with
$\sup_{t,x}\|\nabla b(t,x)\|_{\mathrm{op}} =: L_b<\infty$, and $b$ is Lipschitz
in time uniformly in $x$: $|b(t,x)-b(s,x)|\le L_{b,t}|t-s|$ for some
$L_{b,t}<\infty$. For the reverse-score estimates we further
assume that $\operatorname{div}b(t,\cdot)$ has a weak gradient satisfying
$\sup_{t,x}|\nabla\operatorname{div}b(t,x)|=:G_b<\infty$.
$\sigma:[0,T]\to\mathbb{R}^{d\times m}$ is measurable,
$a(t):=\sigma(t)\sigma(t)^{\mathsf{T}}$, and
$\lambda_0 I_d\preceq a(t)\preceq\Lambda_0 I_d$ for some
$0<\lambda_0\le\Lambda_0<\infty$.
\end{enumerate}
The time-Lipschitz bound on $b$ is used only to control coefficient freezing.
The bound on $\nabla\operatorname{div}b$ is used only in the reverse-score
dynamics and its consequences.

Under (F), the equation~\eqref{eq:forward} admits a unique strong solution
$X$ on $[0,T]$.
For $t\in(0,T]$, let $p_t$ denote the density of $X_t$, and set
$p_0:=p_{data}$.

\begin{assum}[Reverse-score closure]
\label{ass:score_closure}
\begin{enumerate}
\item[(i)] On $(0,T)\times\mathbb R^d$, the function
$(t,x)\mapsto\log p_t(x)$ has continuous spatial derivatives up to order
three, and $\partial_t\log p_t(x)$ and $\partial_t\nabla\log p_t(x)$ exist
and are continuous.
\item[(ii)] Along the forward process,
\[
 \begin{aligned}
 \lim_{s\downarrow0}\mathbb E\big|
 \nabla\log p_s(X_s)-\nabla\log p_{data}(X_0)\big|^2&=0,\\
 \lim_{s\uparrow T}\mathbb E\big|
 \nabla\log p_s(X_s)-\nabla\log p_T(X_T)\big|^2&=0.
 \end{aligned}
\]
\end{enumerate}
\end{assum}
Part (i) permits It\^o's formula to be applied to
$\nabla\log p_{T-t}(X_t^*)$ on compact subintervals of $(0,T)$.  Part (ii)
provides the two endpoint traces needed for the backward stochastic
differential equation (BSDE).  For bounded nonlinear
perturbations these are explicit analytic conditions.
Proposition~\ref{prop:linear_closure} verifies the required conclusions for
the linear VP and VE cases under (D), so their corollaries do not assume
Assumption~\ref{ass:score_closure} separately.

A structural feature of the class is that every transition kernel is
two-sided comparable, with explicit constants, to an explicit Gaussian kernel.
Let $\Phi(s,t)$ be the state-transition matrix of $A$ from $t$ to $s$:
$\partial_s\Phi(s,t)=A(s)\Phi(s,t)$, $\Phi(t,t)=I_d$; note
$N_A:=\int_0^T\|A(t)\|_{\mathrm{op}}dt\le TM_A$ and
$\|\Phi(s,t)\|_{\mathrm{op}}\vee\|\Phi(t,s)\|_{\mathrm{op}}\le e^{N_A}$ for all
$s,t$.
The unperturbed linear diffusion
$d\Xi_s = A(s)\Xi_s\,ds+\sigma(s)\,dW_s$ has the explicit Gaussian kernel
\begin{equation}
\label{eq:gauss_kernel}
 q(t,x,s,y) := \varphi_{V(t,s)}\big(y-\Phi(s,t)x\big),
 \qquad V(t,s):=\int_t^s\Phi(s,u)a(u)\Phi(s,u)^{\mathsf{T}}\,du,
\end{equation}
and (F) gives $\lambda_0 e^{-2N_A}(s-t)I_d\preceq V(t,s)\preceq
\Lambda_0 e^{2N_A}(s-t)I_d$.

\begin{assum}[Kernel estimates]
\label{lem:kernel}
The solution of~\eqref{eq:forward} has a transition density
$p(t,x,s,y)$ for $0\le t<s\le T$, continuously differentiable in $x$ and $y$,
and there exists a constant $c_*\ge 1$ such that the following hold for all
$0\le t<s\le T$ and $x,y\in\mathbb{R}^d$.
\begin{enumerate}
\item[(i)] (Two-sided Gaussian bounds)
\[
 c_*^{-1}\,\varphi_{V(t,s)/c_*}\big(y-\Phi(s,t)x\big)
 \le p(t,x,s,y)
 \le c_*\,\varphi_{c_*V(t,s)}\big(y-\Phi(s,t)x\big).
\]
\item[(ii)] (Log-gradient bound)
\[
 |\nabla_y\log p(t,x,s,y)|
 \le c_*\,\|V(t,s)^{-1/2}\|_{\mathrm{op}}
   \Big(1+\big|V(t,s)^{-1/2}\big(y-\Phi(s,t)x\big)\big|\Big).
\]
\item[(iii)] (Integration-by-parts correction)
There exists a Borel measurable function
$\vartheta:[0,T]\times\mathbb R^d\times[0,T]\times\mathbb R^d
\to\mathbb R^d$ such that
\[
 \nabla_y p(t,x,s,y)
 = -\Phi(s,t)^{-\mathsf{T}}\,\nabla_x p(t,x,s,y)
   + p(t,x,s,y)\,\vartheta(t,x,s,y),
 \qquad |\vartheta(t,x,s,y)|\le c_*,
\]
and $\vartheta\equiv 0$ when $b\equiv 0$.
\end{enumerate}
\end{assum}

The Gaussian bounds in (i)--(ii) follow from standard parametrix estimates
under familiar additional regularity assumptions; see
Friedman~\cite{friedman:1964}, Aronson~\cite{aronson:1967}, and
Sheu~\cite{sheu:1991}.  The endpoint identity (iii) is stronger and is kept
as an explicit analytic hypothesis rather than being inferred from Gaussian
bounds: two-sided Gaussian bounds with dilated covariances do not imply that
$p/q$ is uniformly bounded.  When $b\equiv0$, all three statements hold
directly with $p=q$, $c_*=1$, and $\vartheta\equiv0$.  Applications with
$b\ne0$ must verify (iii) for the chosen coefficients.

For $t\in(0,T]$, the marginal density has the representation
\[
 p_t(y)=\int p(0,x,t,y)\,\mu_{data}(dx).
\]
Let $\mu_t(dy):=p_t(y)\,dy$ denote the corresponding probability measure.
We write $\nabla\log p_t$ for the score and extend it by zero where $p_t$
vanishes.
Assumption~\textup{\ref{lem:kernel}(i)} makes $p_t$ bounded for each $t>0$.
Condition (D) gives $\|\nabla p_{data}\|_{L^1}\le\sqrt{J_1}<\infty$, which
justifies the integration by parts used below.

\begin{rem}[Sufficient conditions for endpoint closure]
\label{rem:endpoint_closure}
The two limits in Assumption~\textup{\ref{ass:score_closure}(ii)} can be checked
separately.  Put $U_0:=\nabla\log p_{data}(X_0)$.  The identity in
Assumption~\textup{\ref{lem:kernel}(iii)} and integration by parts give
\[
 \nabla\log p_s(X_s)
 =\Phi(s,0)^{-\mathsf T}\mathbb E[U_0\mid X_s]
  +\mathbb E[\vartheta(0,X_0,s,X_s)\mid X_s].
\]
Since $X_s\to X_0$ in $L^2$, $\Phi(s,0)\to I_d$, and $U_0\in L^2$, the projection property of
conditional expectation gives
$\mathbb E[U_0\mid X_s]\to U_0$ in $L^2$.  Hence the first endpoint limit
follows from
\begin{equation}
\label{eq:theta_endpoint_sufficient}
 \mathbb E|\vartheta(0,X_0,s,X_s)|^2\longrightarrow0
 \qquad\text{as }s\downarrow0.
\end{equation}
For example, Assumption~\textup{\ref{lem:kernel}(i)} shows that
\eqref{eq:theta_endpoint_sufficient} holds if, for some $\alpha>0$,
\[
 |\vartheta(0,x,s,y)|
 \le C\big(s^\alpha+|y-\Phi(s,0)x|^\alpha\big).
\]

For the second endpoint, it is sufficient that
\begin{equation}
\label{eq:terminal_score_sufficient}
 \nabla\log p_s\longrightarrow\nabla\log p_T
 \quad\text{locally uniformly as }s\uparrow T.
\end{equation}
Indeed, path continuity then gives convergence in probability after
evaluation at $X_s$.  Assumption~\textup{\ref{lem:kernel}(i)--(ii)} gives, for some
$\eta>0$,
\[
 \sup_{T/2\le s\le T}
 \mathbb E|\nabla\log p_s(X_s)|^{2+\eta}<\infty.
\]
Uniform integrability therefore upgrades the convergence to $L^2$.
Condition~\eqref{eq:terminal_score_sufficient} follows, for example, from
local uniform convergence of both $p_s$ and $\nabla p_s$ to $p_T$ and
$\nabla p_T$, since Assumption~\textup{\ref{lem:kernel}(i)} makes $p_T$ positive.
\end{rem}

\begin{rem}
\label{rem:constants}
Generic constants $C$ depend only on
$J_1,M_2,M_A,R_0,L_b,L_{b,t},G_b,\lambda_0,\Lambda_0,c_*$, and~$T$.
\end{rem}

\section{Fisher information estimate}\label{sec:curvature}

The proof combines two representations of the marginal score.  Integration
by parts gives a data-based bound effective near $t=0$.  Direct kernel
differentiation gives a noise-based bound effective away from $t=0$.  Their
minimum yields the uniform estimate below.

\begin{thm}[Uniform Fisher bound]
\label{thm:fisher_bound}
Assume \textup{(D)}, \textup{(F)}, and
Assumption~\textup{\ref{lem:kernel}}.
Then for $t\in(0,T]$,
\begin{equation}
\label{eq:two_regime}
 \mathbb{E}\,|\nabla\log p_t(X_t)|^2
 \le C\,\min\Big\{\;\|\Phi(0,t)\|_{\mathrm{op}}^2\,(J_1+1),\;\;
 (1+d)\,\|V(0,t)^{-1}\|_{\mathrm{op}}\;\Big\},
\end{equation}
and consequently
\begin{equation}
\label{eq:fisher_uniform}
 \bar{J}:=\sup_{t\in(0,T]}\mathbb{E}\,|\nabla\log p_t(X_t)|^2
 \le C\,\sup_{t\in(0,T]}\min\Big\{\|\Phi(0,t)\|_{\mathrm{op}}^2(J_1+1),\,
   (1+d)\|V(0,t)^{-1}\|_{\mathrm{op}}\Big\}<\infty.
\end{equation}
\end{thm}
\begin{proof}
Assumption~\textup{\ref{lem:kernel}(i)--(ii)} justifies differentiating
$p_t(y)=\int p(0,x,t,y)\,p_{data}(x)\,dx$ under the integral sign.
Since $\|\nabla p_{data}\|_{L^1}<\infty$, we may then use
Assumption~\textup{\ref{lem:kernel}(iii)} and integrate by parts in~$x$ to obtain
\begin{align*}
 \nabla p_t(y)
 &= \int\big[-\Phi(t,0)^{-\mathsf{T}}\nabla_x p(0,x,t,y)
   + p(0,x,t,y)\,\vartheta(0,x,t,y)\big]\,p_{data}(x)\,dx \\
 &= \Phi(t,0)^{-\mathsf{T}}\!\int p(0,x,t,y)\,\nabla p_{data}(x)\,dx
   + \int p(0,x,t,y)\,\vartheta(0,x,t,y)\,p_{data}(x)\,dx,
\end{align*}
whence
\[
 \nabla\log p_t(y)
 = \Phi(t,0)^{-\mathsf{T}}\,
   \mathbb{E}\big[\nabla\log p_{data}(X_0)\,\big|\,X_t=y\big]
   + \mathbb{E}\big[\vartheta(0,X_0,t,y)\,\big|\,X_t=y\big].
\]
By Jensen's inequality, $\Phi(t,0)^{-1}=\Phi(0,t)$, $|\vartheta|\le c_*$,
and (D),
\[
 \mathbb{E}|\nabla\log p_t(X_t)|^2
 \le 2\|\Phi(0,t)\|_{\mathrm{op}}^2\,J_1 + 2c_*^2
 \le C\,\|\Phi(0,t)\|_{\mathrm{op}}^2\,(J_1+1).
\]

Alternatively,
$\nabla\log p_t(y)=\mathbb{E}[\nabla_y\log p(0,X_0,t,y)\,|\,X_t=y]$,
and by Jensen and Assumption~\textup{\ref{lem:kernel}(ii)},
\[
 \mathbb{E}|\nabla\log p_t(X_t)|^2
 \le 2c_*^2\,\|V(0,t)^{-1}\|_{\mathrm{op}}
   \Big(1+\mathbb{E}\big|V(0,t)^{-1/2}\big(X_t-\Phi(t,0)X_0\big)\big|^2\Big).
\]
By Assumption~\textup{\ref{lem:kernel}(i)}, conditionally on $X_0$ the random variable
$V(0,t)^{-1/2}(X_t-\Phi(t,0)X_0)$ has second moment bounded by
$c_*^2d$ (Gaussian moment under the upper kernel bound).  Thus the
parenthesis above is bounded by $1+c_*^2d$, and the
right-hand side is at most $C(1+d)\|V(0,t)^{-1}\|_{\mathrm{op}}$.
Combining the two bounds proves~\eqref{eq:two_regime}.  Taking the supremum
and using $\|\Phi(0,t)\|_{\mathrm{op}}\le e^{N_A}$ proves
\eqref{eq:fisher_uniform}.
\end{proof}

\begin{rem}[DDPM evaluation]
\label{rem:DDPM_fisher}
For DDPM ($A(t)=-\tfrac{1}{2}\beta_tI_d$, $b\equiv 0$,
$a(t)=\beta_tI_d$, $T=1$), one has
$\Phi(0,t)=m_{0,t}^{-1}I_d$ and $V(0,t)=(1-m_{0,t}^2)I_d=\sigma_{0,t}^2I_d$,
where $m_{0,t}=e^{-\frac12\int_0^t\beta_r\,dr}$.
To combine the two bounds, split according to the signal level.
If $m_{0,t}^2\ge\tfrac12$, the bound from the data representation gives
$\mathbb{E}|\nabla\log p_t(X_t)|^2\le2C(J_1+1)$.  If
$m_{0,t}^2<\tfrac12$, then $\sigma_{0,t}^2=1-m_{0,t}^2>\tfrac12$, and the
bound from the kernel representation gives
$\mathbb{E}|\nabla\log p_t|^2\le2C(1+d)$.
Hence
\[
 \bar{J}\le C\,(J_1+d),
\]
uniformly in the noise schedule.
The data representation alone carries the factor
$\|\Phi(0,t)\|_{\mathrm{op}}^2=m_{0,t}^{-2}$.  It grows like
$(\bar{\alpha}_n)^{-1}$ at the end of the schedule.  The kernel representation
removes this growth.  Taking the smaller of the two bounds gives a constant
that is uniform in the schedule.
\end{rem}

\section{Discrete-time sampler and convergence analysis}\label{sec:sampler}

\subsection{Reverse-time SDE}

Define the exact reverse process by
$X_t^*:=X_{T-t}$, $0\le t\le T$.  Then
$X_0^*\sim\mu_T$ and $X_T^*\sim\mu_{data}$.
Write $\mathbb E^*$ for expectation under the law of $X^*$.
Under (D), (F), and Assumption~\ref{lem:kernel},
Theorem~\ref{thm:fisher_bound} gives
\[
 \int_0^T\mathbb E|\sigma(s)^{\mathsf T}\nabla\log p_s(X_s)|^2ds
 \le \Lambda_0\bar JT<\infty.
\]
Hence the time-reversal theorem of Haussmann and Pardoux~\cite{hau-par:1986}
applies.  It yields an $m$-dimensional Brownian motion $W^*$ for a reverse
filtration.  This Brownian motion is not identified with the forward Brownian
motion $W$.  Since $a=a(t)$, the divergence term vanishes, and $X^*$ satisfies
\begin{equation}
\label{eq:reverse}
 dX_t^* = \big[-A(T-t)X_t^*-b(T-t,X_t^*)
   + a(T-t)\,\nabla\log p_{T-t}(X_t^*)\big]\,dt + \sigma(T-t)\,dW^*_t.
\end{equation}
Its score dynamics follows directly from the
Fokker--Planck equation and includes a divergence correction for nonlinear
$b$.

For $t<T$, define $Y_t:=\nabla\log p_{T-t}(X_t^*)$ and
$Z_t:=\nabla^2\log p_{T-t}(X_t^*)\,\sigma(T-t)$.  Set
$Y_T:=\nabla\log p_{data}(X_T^*)$ and $Z_T:=0$.
\nopagebreak[4]

\begin{thm}[Reverse-score dynamics and energy]
\label{thm:reverse_energy}
Assume \textup{(D)}, \textup{(F)},
Assumption~\textup{\ref{ass:score_closure}}, and
Assumption~\textup{\ref{lem:kernel}}.  Then $(Y,Z)$ is the
unique square-integrable solution of
\begin{equation}
\label{eq:reverse_score_bsde}
 dY_t=\Big[(A(T-t)+\nabla b(T-t,X_t^*))^{\mathsf T}Y_t
 +\nabla\operatorname{div}b(T-t,X_t^*)\Big]dt+Z_t\,dW_t^*.
\end{equation}
Equivalently, with the bracketed drift denoted by $F_t$,
\[
 Y_t=Y_T-\int_t^T F_r\,dr-\int_t^T Z_r\,dW_r^*,
 \qquad 0\le t\le T.
\]
In particular, its terminal value is the data score and
\begin{equation}
\label{eq:reverse_Z_energy}
 \mathbb E^*\!\int_0^T|Z_t|^2dt
 \le \mathbb E^*|Y_T|^2
 +(1+2M_A+2L_b)\int_0^T\mathbb E^*|Y_t|^2dt+T G_b^2.
\end{equation}
\end{thm}
The convention $Z_T=0$ only fixes the value at one time and does not affect
the stochastic integral or the energy estimate.
\begin{proof}
Write $s=T-t$, $f_s(x)=A(s)x+b(s,x)$, and
$v(t,x)=\log p_{T-t}(x)$.  Dividing the Fokker--Planck equation by $p_s$
gives
\[
 \partial_t v=\operatorname{div}f_s+f_s\cdot\nabla v
 -\tfrac12\tr(a(s)\nabla^2v)-\tfrac12|\nabla v|_{a(s)}^2.
\]
Take the spatial gradient of this equation and apply It\^o's formula to
$\nabla v(t,X_t^*)$.  After substituting the reverse drift
from~\eqref{eq:reverse}, the Hessian transport terms cancel, as do the
third-order terms from the differentiated PDE and the diffusion term in
It\^o's formula.  The remaining terms give
\[
 dY_t=\big[(\nabla f_s(X_t^*))^{\mathsf T}Y_t
       +\nabla\operatorname{div}f_s(X_t^*)\big]dt+Z_t\,dW_t^*.
\]
Since $\nabla f_s=A(s)+\nabla b(s,\cdot)$ and
$\nabla\operatorname{div}f_s=\nabla\operatorname{div}b(s,\cdot)$, this is
\eqref{eq:reverse_score_bsde} locally on $(0,T)$.
Exact time reversal gives
\[
 (X^*_{T-s},X_T^*)\stackrel d=(X_s,X_0),\qquad
 (X^*_{T-s},X_0^*)\stackrel d=(X_s,X_T).
\]
Hence Assumption~\textup{\ref{ass:score_closure}(ii)} identifies the two
$L^2$ endpoint traces of $Y$, including
$Y_T=\nabla\log p_{data}(X_T^*)$.

Fix $0<\varepsilon<T/2$.  Apply It\^o's formula to $|Y_t|^2$ on
$[\varepsilon,T-\varepsilon]$ and localize.  The bounds in (F) give
\[
 \mathbb E^*\!\int_\varepsilon^{T-\varepsilon}|Z_t|^2dt
 \le \mathbb E^*|Y_{T-\varepsilon}|^2
 +(1+2M_A+2L_b)\int_\varepsilon^{T-\varepsilon}
   \mathbb E^*|Y_t|^2dt+T G_b^2.
\]
Here we used $-2y^{\mathsf T}(A+\nabla b)^{\mathsf T}y
\le2(M_A+L_b)|y|^2$, $-2y\cdot g\le |y|^2+|g|^2$, and dropped the
nonpositive initial term.  Assumption~\textup{\ref{ass:score_closure}(ii)} identifies
the $L^2$ limits of $Y$ at both endpoints.  Letting
$\varepsilon\downarrow0$, Fatou's lemma on the left-hand side and monotone
convergence in the $Y$-integral give~\eqref{eq:reverse_Z_energy}.  Exact time
reversal and Theorem~\ref{thm:fisher_bound} give
\[
 \int_0^T\mathbb E^*|Y_t|^2dt\le T\bar J,
 \qquad \mathbb E^*|Y_T|^2=J_1.
\]
Thus the right-hand side of~\eqref{eq:reverse_Z_energy} is finite, and hence
$\mathbb E^*\int_0^T|Z_t|^2dt<\infty$.

The drift integrals then converge in $L^1$ because $A+\nabla b$ and
$\nabla\operatorname{div}b$ are bounded.  The stochastic integrals converge
in $L^2$ by the square integrability of $Z$.  Together with the endpoint
limits of $Y$, this extends the integrated equation from compact
subintervals of $(0,T)$ to $[0,T]$.  Uniqueness follows by applying the usual
energy estimate to the difference of two solutions.
\end{proof}

The next result supplies the reverse-score identity and endpoint closure for
the linear VP and VE specializations.  It therefore replaces
Assumption~\ref{ass:score_closure} in these cases.

\begin{prop}[Closure for the VP and VE diffusions]
\label{prop:linear_closure}
Suppose \textup{(D)}, $b\equiv0$, and either
\begin{enumerate}
\item[\textup{(i)}] $A(t)=-\tfrac12\beta(t)I_d$ and $a(t)=\beta(t)I_d$, or
\item[\textup{(ii)}] $A(t)=0$ and $a(t)=\gamma(t)I_d$,
\end{enumerate}
where the scalar schedule is bounded, bounded away from zero, and piecewise
continuous with finitely many discontinuities.  Let
$0=\tau_0<\cdots<\tau_K=T$ be a partition such that the schedule is
continuous on each $(\tau_j,\tau_{j+1})$, $j=0,\ldots,K-1$.  Then the endpoint condition in
Assumption~\textup{\ref{ass:score_closure}(ii)} holds.  The reverse-score
identity~\eqref{eq:reverse_score_bsde} holds on each
$(T-\tau_{j+1},T-\tau_j)$, $j=0,\ldots,K-1$, and its integrated form holds
on $[0,T]$.  The energy bound \eqref{eq:reverse_Z_energy} holds along the
exact reverse process started from $\mu_T$.
\end{prop}
\begin{proof}
Write $X_s=m_sX_0+v_s^{1/2}Z$, where $m_s>0$ and $v_s>0$ for $s>0$;
$m_s=e^{-\frac12\int_0^s\beta_rdr}$ in case~(i), and $m_s=1$ in case~(ii).
The functions $m$ and $v$ are continuous on $[0,T]$ and continuously
differentiable on each $(\tau_j,\tau_{j+1})$.  Thus $p_s$ is smooth and
positive for $s>0$, and the local calculation in the proof of
Theorem~\ref{thm:reverse_energy} applies on each corresponding reverse-time
interval.  The continuity of $m$ and $v$ joins these identities across the
partition points.

Put $U_0:=\nabla\log p_{data}(X_0)\in L^2$ by (D).  Let $g_v$ denote the
density of $N(0,vI_d)$.  For $s>0$,
\[
 p_s(y)=\int_{\mathbb R^d}g_{v_s}(y-m_sx)p_{data}(x)\,dx.
\]
Since
$\nabla_xg_{v_s}(y-m_sx)=-m_s\nabla_yg_{v_s}(y-m_sx)$, weak integration by
parts, justified by $\|\nabla p_{data}\|_{L^1}\le J_1^{1/2}$, gives
\[
 \nabla p_s(y)=m_s^{-1}\int_{\mathbb R^d}
 g_{v_s}(y-m_sx)\nabla p_{data}(x)\,dx,
 \qquad
 \nabla\log p_s(X_s)=m_s^{-1}\mathbb E[U_0\mid X_s].
\]
As $s\downarrow0$, $X_s\to X_0$ in $L^2$ and $m_s\to1$.  Approximating
$U_0$ in $L^2$ by bounded continuous functions of $X_0$ and using the
projection property of conditional expectation therefore gives
$\nabla\log p_s(X_s)\to U_0$ in $L^2$.

As $s\uparrow T$, the Gaussian convolution formula gives
$\nabla\log p_s\to\nabla\log p_T$ locally uniformly, while path continuity
gives $X_s\to X_T$ almost surely.  Moreover, for any $q=2+\eta>2$, the
Gaussian score identity and conditional Jensen's inequality give
$\sup_{T/2\le s\le T}\mathbb E|\nabla\log p_s(X_s)|^q
\le(\inf_{T/2\le s\le T}v_s)^{-q/2}\mathbb E|Z|^q<\infty$.
Hence uniform integrability upgrades the almost-sure convergence to $L^2$,
and Assumption~\textup{\ref{ass:score_closure}(ii)} holds.

Under exact reversal, $(X^*_{T-s},X_T^*)\stackrel d=(X_s,X_0)$, hence
$Y_{T-s}\to U_0=Y_T$ in $L^2$.  The energy identity on $[0,T-\delta]$ gives
\[
 \mathbb E^*\!\int_0^{T-\delta}|Z_t|^2dt
 \le \mathbb E^*|Y_{T-\delta}|^2
   +2M_A\int_0^{T-\delta}\mathbb E^*|Y_t|^2dt.
\]
The right-hand side is bounded uniformly in $\delta$ by
Theorem~\ref{thm:fisher_bound}.  Thus the stochastic integrals converge in
$L^2$ and the drift integrals by dominated convergence.  Letting
$\delta\downarrow0$ proves the integrated identity and
\eqref{eq:reverse_Z_energy}.
\end{proof}

\subsection{Exponential-integrator scheme}

Let $\Psi(s,t)$ denote the state-transition matrix of $-A(T-\cdot)$:
$\partial_s\Psi(s,t)=-A(T-s)\Psi(s,t)$, $\Psi(t,t)=I_d$.
Partition $[0,T]$ as $0=t_0<\cdots<t_n=T$ with $\Delta_i:=t_{i+1}-t_i$,
$\Delta_{\max}:=\max_i\Delta_i$, and let
$\mathbf{s}_i:\mathbb{R}^d\to\mathbb{R}^d$ be the Borel measurable learned score used on the
$i$-th step.
The sampler integrates the linear part of~\eqref{eq:reverse} \emph{exactly}
and freezes the score and the bounded perturbation at the left endpoint:
\begin{equation}
\label{eq:discrete}
\begin{aligned}
 \hat{X}_{t_{i+1}}
 &= \Psi(t_{i+1},t_i)\hat{X}_{t_i}
  + \Big(\int_{t_i}^{t_{i+1}}\!\Psi(t_{i+1},s)\,a(T-s)\,ds\Big)\,
    \mathbf{s}_{i+1}(\hat{X}_{t_i}) \\
 &\quad
  - \Big(\int_{t_i}^{t_{i+1}}\!\Psi(t_{i+1},s)\,ds\Big)\,
    b(T-t_i,\hat{X}_{t_i})
  + \xi_{i+1},
\end{aligned}
\end{equation}
with $\hat{X}_0\sim\nu$ and independent Gaussian innovations
$\xi_{i+1}\sim N(0,\Sigma_i)$, where
\[
 \Sigma_i:=\int_{t_i}^{t_{i+1}}
 \Psi(t_{i+1},s)a(T-s)\Psi(t_{i+1},s)^{\mathsf{T}}ds.
\]
For an initial law $\rho$, write $\hat\mu_\rho$ for the output law of this
fixed numerical transition kernel when $\hat X_0\sim\rho$, and abbreviate
$\hat\mu:=\hat\mu_\nu$ for the practical initialization.

\begin{rem}[Relation to DDPM]
\label{rem:ddpm_exact}
For a DDPM schedule $\{\alpha_i\}_{i=1}^n\subset(0,1)$ on $t_i=i/n$, put
$\bar\alpha_0=1$, $\bar\alpha_i=\prod_{k=1}^i\alpha_k$, and let $g$ interpolate
$-\log\bar\alpha_i$.  Taking $b=0$, $A(t)=-\tfrac12g'(t)I_d$,
$a(t)=g'(t)I_d$, and $T=1$, formula~\eqref{eq:discrete} becomes
\[
 \hat{X}_{t_{i+1}}=\alpha_{n-i}^{-1/2}\hat{X}_{t_i}
 +2(\alpha_{n-i}^{-1/2}-1)\,
    \mathbf{s}_{i+1}(\hat{X}_{t_i})+\hat\xi_{i+1},
\]
where $\hat\xi_{i+1}\sim N(0,(\alpha_{n-i}^{-1}-1)I_d)$.
In noise prediction, $\mathbf{s}_{i+1}=-z_{n-i}/\sqrt{1-\bar\alpha_{n-i}}$,
where $z_j$ predicts the forward noise at time $t_j$.
This is the VP exponential integrator.  Its score coefficient
$2(\alpha_{n-i}^{-1/2}-1)$ differs from the ancestral-DDPM coefficient
$\alpha_{n-i}^{-1/2}(1-\alpha_{n-i})$, and its innovation variance need not
match an ancestral-DDPM convention.  Since $b=0$, score freezing is the sole
discretization.
\end{rem}

\subsection{Learning objective}

The score fields in the sampler are indexed in reverse time:
$\mathbf s_{i+1}$ approximates the forward score at
$r_i:=T-t_i$.  The corresponding unweighted marginal error is
\begin{equation}
\label{eq:Lscore}
 L_{\mathrm{rev}} := \frac1T\sum_{i=0}^{n-1}\Delta_i
   \mathbb{E}\big|\mathbf{s}_{i+1}(X_{r_i})
   -\nabla\log p_{r_i}(X_{r_i})\big|^2.
\end{equation}
The score is unknown, but it is the conditional mean of the transition
score.  Orthogonal projection therefore gives the denoising
objective~\cite{vincent:2011}
\begin{equation*}
 L^{\mathrm{den}}_{\mathrm{rev}}(\mathbf{s}):=\frac1T\sum_{i=0}^{n-1}\Delta_i
   \mathbb{E}\big|\mathbf{s}_{i+1}(X_{r_i})
   -\nabla_y\log p(0,X_0,r_i,X_{r_i})\big|^2
   = L_{\mathrm{rev}} + \frac1T\sum_{i=0}^{n-1}\Delta_iC_i,
\end{equation*}
where $C_i:=\mathbb E|\nabla_y\log p(0,X_0,r_i,X_{r_i})
-\nabla\log p_{r_i}(X_{r_i})|^2$ is independent of $\mathbf s_{i+1}$.
Thus both objectives have the same minimizers.  For $b=0$ the target is
explicit by~\eqref{eq:gauss_kernel}.  For nonlinear $b$ our theorem
uses~\eqref{eq:Lscore}.

\subsection{Initial-law error for the exact reverse kernel}\label{sec:initial}

The next result bounds the error from replacing the exact initial law $\mu_T$
by $\nu$.

\begin{thm}[Initial-law error]
\label{thm:initial}
Under Assumption~\textup{\ref{lem:kernel}}, define the
reverse terminal kernel by
\[
 \overleftarrow K(x,dy)
 :=\frac{p_{data}(y)p(0,y,T,x)}{p_T(x)}\,dy.
\]
For a probability measure $\nu$, write
$(\nu\overleftarrow K)(A):=\int_{\mathbb{R}^d}\overleftarrow K(x,A)\,\nu(dx)$
for each Borel set $A\subset\mathbb{R}^d$.
Then, for any $\nu\in\mathcal{P}(\mathbb{R}^d)$ with
$D_{KL}(\nu\,\|\,\mu_T)<\infty$,
\begin{equation}
\label{eq:initial_TV}
 D_{TV}\big(\mu_{data},\nu\overleftarrow K\big)
 \le D_{TV}(\nu,\mu_T)
 \le \sqrt{\tfrac{1}{2}\,D_{KL}(\nu\,\|\,\mu_T)}.
\end{equation}
\end{thm}
\begin{proof}
Assumption~\textup{\ref{lem:kernel}(i)} gives $p_T(x)>0$.  The definition of $p_T$
shows that $\overleftarrow K$ is a Markov kernel and that
$\mu_T\overleftarrow K=\mu_{data}$.  Total variation contracts under Markov
kernels.  Hence
\[
 D_{TV}\big(\mu_{data},\nu\overleftarrow K\big)
 = D_{TV}\big(\mu_T\,\overleftarrow{K},\,\nu\,\overleftarrow{K}\big)
 \le D_{TV}(\mu_T,\nu).
\]
Pinsker's inequality gives
$D_{TV}(\nu,\mu_T)\le\sqrt{D_{KL}(\nu\,\|\,\mu_T)/2}$
and hence~\eqref{eq:initial_TV}; see, e.g., Tsybakov~\cite{tsy:2009}.
\end{proof}

\begin{rem}[Explicit bound on the initialization mismatch]
\label{rem:initial_ddpm}
For the VP process with $\nu=N(0,I_d)$, write
$m_T=e^{-\frac12\int_0^T\beta_sds}$ and $v_T=1-m_T^2$. Then $\mu_T$ is the law
of $m_T\mathbf{x}_0+\sqrt{v_T}Z$, where $Z\sim N(0,I_d)$ is independent of
$\mathbf x_0$.  Jensen's inequality for $-\log$ gives
\[
\begin{aligned}
 D_{KL}\big(N(0,I_d)\,\|\,\mu_T\big)
 &\le \mathbb E\,D_{KL}\big(N(0,I_d)\,\|\,N(m_T\mathbf x_0,v_TI_d)\big)\\
 &=\frac d2\big(v_T^{-1}-1+\log v_T\big)
   +\frac{m_T^2M_2}{2v_T}
 \le \frac{m_T^2}{2v_T}(d+M_2),
\end{aligned}
\]
where the last inequality uses $v_T=1-m_T^2$ and $\log v_T\le0$.
Thus the initialization contribution in total variation is $O(m_T)$ when
$v_T$ stays away from zero.
\end{rem}

\subsection{Score and discretization error}\label{sec:score_disc}

The following theorem bounds the score and time-discretization errors relative
to the exact reverse process.  Its proof combines the reverse-score energy
estimate with the uniform Fisher bound.
For $t\in[t_i,t_{i+1})$, set
\[
 \tau_n(t):=t_i,\qquad
 Y_t:=\nabla\log p_{T-t}(X_t^*),\qquad
 \theta_{t_i}:=Y_{t_i}-\mathbf{s}_{i+1}(X_{t_i}^*),
\]
and
\[
 \Delta_t^b:=b(T-\tau_n(t),X^*_{\tau_n(t)})-b(T-t,X_t^*).
\]

\begin{thm}[Main estimate: score and discretization error]
\label{thm:disc_score}
Assume \textup{(D)}, \textup{(F)},
Assumption~\textup{\ref{ass:score_closure}}, and
Assumption~\textup{\ref{lem:kernel}}.  Suppose
$\mathcal{L}(X_0^*)=\mathcal{L}(\hat X_0)=\mu_T$ and
$L_{\mathrm{rev}}<\infty$.
Define
\begin{align*}
 R_{\mathrm{score}}
 &:=4\Lambda_0\sum_{i=0}^{n-1}\Delta_i\,
       \mathbb E^*|\theta_{t_i}|^2,\\
 R_{\mathrm{disc}}
 &:=2\lambda_0^{-1}\mathbb E^*\!\int_0^T|\Delta_t^b|^2\,dt
   +4\Lambda_0\mathbb E^*\!\int_0^T|Y_t-Y_{\tau_n(t)}|^2\,dt.
\end{align*}
Then
\begin{equation}
\label{eq:disc_score_bound}
 D_{TV}\big(\mu_{data},\hat{\mu}_{\mu_T}\big)
 \le \tfrac{1}{2}\sqrt{R_{\mathrm{score}} + R_{\mathrm{disc}}},
\end{equation}
\begin{equation}
\label{eq:R_score}
 R_{\mathrm{score}}
 =4\Lambda_0\,T L_{\mathrm{rev}},
\end{equation}
\begin{equation}
\label{eq:R_disc}
\begin{aligned}
 R_{\mathrm{disc}}
 &\le C\,\Delta_{\max}\Big\{
 \Lambda_0\big[J_1+(1+M_A+L_b)T\bar J+TG_b^2\big]
 +\lambda_0^{-1}L_b^2(1+\bar J)T\Big\}
 \\
 &\quad+C\Lambda_0\big[(M_A+L_b)^2\bar J+G_b^2\big]
 T\Delta_{\max}^2+C\lambda_0^{-1}L_{b,t}^2T\Delta_{\max}^2.
\end{aligned}
\end{equation}
\end{thm}
\begin{proof}
Write $s_n(t,\cdot):=\mathbf{s}_{i+1}$ for $t\in[t_i,t_{i+1})$, and let
$Z$ be the martingale integrand from Theorem~\ref{thm:reverse_energy}.

Step 1.
As noted in Remark~\ref{rem:ddpm_exact}, the exponential-integrator
scheme~\eqref{eq:discrete} is the time-$t_i$ skeleton of the continuous
process $\hat{X}$ solving the same equation as~\eqref{eq:reverse} but with the
score and the bounded drift frozen at the left endpoints:
\[
 d\hat{X}_t = \big[-A(T-t)\hat{X}_t - b(T-\tau_n(t),\hat{X}_{\tau_n(t)})
   + a(T-t)\,s_n(t,\hat{X}_{\tau_n(t)})\big]\,dt + \sigma(T-t)\,dW_t^*,
\]
  started at $\hat{X}_0\sim \mu_T$. The processes $X^*$ and $\hat{X}$ share the
  initial law $\mu_T$, the linear drift $-A(T-t)x$, and the diffusion coefficient
$\sigma(T-t)$.
Let $P^*,\hat{P}$ denote their laws on $C([0,T];\mathbb{R}^d)$.
Set $\sigma(t)^\dagger:=\sigma(t)^{\mathsf T}a(t)^{-1}$, the minimum-norm
right inverse of $\sigma(t)$.  Then
$|\sigma(t)^\dagger v|^2=v^{\mathsf T}a(t)^{-1}v$ and
$\sigma(t)^\dagger a(t)=\sigma(t)^{\mathsf T}$.
Because the two diffusion coefficients coincide, a localized Girsanov
argument gives
\begin{equation}
\label{eq:girKL}
 D_{KL}(P^*\,\|\,\hat{P})
 \le \tfrac12\,\mathbb{E}^*\!\int_0^T\big|\sigma(T-t)^\dagger\Delta_t\big|^2\,dt,
 \qquad \Delta_t := \Delta_t^b + a(T-t)\,\Delta_t^s,
\end{equation}
where, along $X^*$,
\[
 \Delta_t^s := \nabla\log p_{T-t}(X^*_t)-s_n(t,X^*_{\tau_n(t)}).
\]
Indeed, localize along
$T_N:=\inf\{t\ge 0:\int_0^t|\sigma(T-r)^\dagger\Delta_r|^2\,dr\ge N\}$ followed by
$N\to\infty$ and use lower semicontinuity of relative entropy.  The limiting
integral is finite by Steps~3--4 below.  This is the same localization
argument used in~\cite[Lemma~5.2]{nakano:2026kernel}; see also
Liptser--Shiryaev~\cite[Ch.~7]{lip-shi:2001}.
Since the terminal marginal is a pushforward, the data-processing inequality
and Pinsker give
\begin{equation}
\label{eq:girsanov_step}
 D_{TV}\big(\mu_{data},\hat{\mu}_{\mu_T}\big)
 \le D_{TV}(P^*,\hat{P})
 \le \sqrt{\tfrac12 D_{KL}(P^*\|\hat{P})}
 \le \tfrac12\sqrt{\mathbb{E}^*\!\int_0^T\big|\sigma(T-t)^\dagger\Delta_t\big|^2\,dt}.
\end{equation}
Using
$|\sigma^\dagger v|^2\le\lambda_0^{-1}|v|^2$,
$|\sigma^{\mathsf{T}}v|^2=v^{\mathsf{T}}av\le\Lambda_0|v|^2$, and
$|u+v|^2\le2|u|^2+2|v|^2$,
\begin{equation}
\label{eq:split_drift}
 \big|\sigma(T-t)^\dagger\Delta_t\big|^2
 \le 2\lambda_0^{-1}\,|\Delta_t^b|^2 + 2\Lambda_0\,|\Delta_t^s|^2 .
\end{equation}

Step 2.
Split $\Delta_t^s$ into two differences,
\[
 \Delta_t^s = \big(Y_t-Y_{\tau_n(t)}\big)
   + \big(\nabla\log p_{T-\tau_n(t)}(X^*_{\tau_n(t)})-s_n(t,X^*_{\tau_n(t)})\big),
\]
the first being the within-step variation of the score (the discretization
part) and the second the score-estimation error at the grid point. Hence
$|\Delta_t^s|^2\le 2|Y_t-Y_{\tau_n(t)}|^2 + 2|\theta_{\tau_n(t)}|^2$,
where $\theta_{t_i}:=\nabla\log p_{T-t_i}(X^*_{t_i})-\mathbf{s}_{i+1}(X^*_{t_i})$.
Thus~\eqref{eq:split_drift} and the definitions above give
$\mathbb E^*\int_0^T|\sigma(T-t)^\dagger\Delta_t|^2dt
\le R_{\mathrm{score}}+R_{\mathrm{disc}}$.

Step 3.
By~\eqref{eq:reverse_score_bsde}, for $t\in[t_i,t_{i+1})$,
\[
 Y_t-Y_{t_i}
 = \int_{t_i}^t \Big[\big(A(T-r)+\nabla b(T-r,X^*_r)\big)^{\mathsf{T}}Y_r
 +\nabla\operatorname{div}b(T-r,X_r^*)\Big]dr
   + \int_{t_i}^t Z_r\,dW_r^*,
\]
whence, with $\|A(T-r)+\nabla b(T-r,\cdot)\|_{\mathrm{op}}\le\|A\|_{L^\infty}+L_b=:\Theta$
and Cauchy--Schwarz in the $dr$-integral,
\[
 \mathbb{E}^*|Y_t-Y_{t_i}|^2
 \le 3\Theta^2(t-t_i)\!\int_{t_i}^t\!\mathbb{E}^*|Y_r|^2\,dr
   +3G_b^2(t-t_i)^2
   + 2\!\int_{t_i}^t\!\mathbb{E}^*|Z_r|^2\,dr .
\]
Integrating over $t\in[t_i,t_{i+1}]$ and summing, we have
\begin{equation}
\label{eq:disc_sum}
 \mathbb{E}^*\!\int_0^T|Y_t-Y_{\tau_n(t)}|^2\,dt
 \le 3\Theta^2\Delta_{\max}^2\!\int_0^T\!\mathbb{E}^*|Y_r|^2\,dr
   +3TG_b^2\Delta_{\max}^2
   +2\Delta_{\max}\,\mathbb{E}^*\!\int_0^T\!|Z_r|^2\,dr .
\end{equation}
We call the drift terms on the right Term~A and the martingale term Term~B.
For Term~A, exact time reversal gives
$X_r^*\sim p_{T-r}$, and hence Theorem~\ref{thm:fisher_bound} gives
$\mathbb{E}^*|Y_r|^2\le\bar J$ for every $r$ and
$\int_0^T\mathbb{E}^*|Y_r|^2\,dr\le\bar J T$. Therefore
\[
\text{Term A}\le C(\Theta^2\bar{J}+G_b^2)T\Delta_{\max}^2 .
\]
For Term~B, the reverse-score energy estimate
(Theorem~\ref{thm:reverse_energy}), exact time reversal, and (D) give
\[
 \mathbb{E}^*\!\int_0^T|Z_r|^2\,dr
 \le J_1+(1+2M_A+2L_b)T\bar J+TG_b^2,
\]
so
\[
 \text{Term B}\le C\Delta_{\max}
 \big[J_1+(1+M_A+L_b)T\bar J+TG_b^2\big].
\]
By the spatial Lipschitz bound $L_b$ and the time-Lipschitz bound $L_{b,t}$
in (F),
\[
 |\Delta_t^b|^2
 \le 2L_b^2|X^*_t-X^*_{\tau_n(t)}|^2
      +2L_{b,t}^2|t-\tau_n(t)|^2 .
\]
The exact reverse SDE~\eqref{eq:reverse} has drift bounded in
$L^2(\mathbb{P}^*)$ by $C(1+\|A\|_{L^\infty}+R_0+\Lambda_0\bar{J}^{1/2})$ and
diffusion bounded by $\Lambda_0$, so a standard one-step estimate gives
$\mathbb{E}^*|X^*_t-X^*_{\tau_n(t)}|^2\le C(1+\bar{J})\,\Delta_{\max}$ for
$t\in[t_i,t_{i+1})$, whence
\begin{equation}
\label{eq:bfreeze}
 \mathbb{E}^*\!\int_0^T|\Delta_t^b|^2\,dt
 \le C\,L_b^2(1+\bar{J})T\Delta_{\max}
      +C\,L_{b,t}^2T\Delta_{\max}^2 .
\end{equation}
Combining~\eqref{eq:disc_sum}--\eqref{eq:bfreeze} with the definition of
$R_{\mathrm{disc}}$ and
$\Theta^2\le2(\|A\|_{L^\infty}^2+L_b^2)$ gives
\[
\begin{aligned}
 R_{\mathrm{disc}}
 &\le C\Delta_{\max}\Big\{
 \Lambda_0\big[J_1+(1+M_A+L_b)T\bar J+TG_b^2\big]
 +\lambda_0^{-1}L_b^2(1+\bar J)T\Big\}
 \\
 &\quad+C\Lambda_0\big[(M_A+L_b)^2\bar J+G_b^2\big]
 T\Delta_{\max}^2+C\lambda_0^{-1}L_{b,t}^2T\Delta_{\max}^2,
\end{aligned}
\]
which is~\eqref{eq:R_disc}.

Step 4.
Since $X^*_{t_i}\sim \mu_{T-t_i}$, the definition of $L_{\mathrm{rev}}$
gives $R_{\mathrm{score}}=4\Lambda_0 T L_{\mathrm{rev}}$, which
is~\eqref{eq:R_score}.
Inserting this and~\eqref{eq:R_disc}
into~\eqref{eq:girsanov_step} gives~\eqref{eq:disc_score_bound}.
\end{proof}

\subsection{End-to-end convergence bounds}\label{sec:main}

The central estimate of Theorem~\ref{thm:disc_score} starts both paths from
$\mu_T$.  To use a tractable initialization $\nu$, let $\widehat K(x,\cdot)$
be the $n$-step transition kernel from time $0$ to $T$ generated by
\eqref{eq:discrete}, with the score fields fixed.  Thus
$\rho\widehat K=\hat\mu_\rho$ for every initial law $\rho$.  Therefore
$D_{TV}(\mu_T\widehat K,\nu\widehat K)\le D_{TV}(\mu_T,\nu)$.  This separates
initialization from the path-space argument and avoids any comparison between
the marginals of reverse processes started from different laws.

\begin{cor}[Total variation convergence]
\label{cor:main}
Suppose \textup{(D)}, \textup{(F)}, Assumption~\textup{\ref{ass:score_closure}}, and
Assumption~\textup{\ref{lem:kernel}} hold. Let
$\nu\in\mathcal{P}(\mathbb{R}^d)$ satisfy $D_{KL}(\nu\,\|\,\mu_T)<\infty$ with
finite second moment, and assume $L_{\mathrm{rev}}<\infty$. Then there is a constant $C>0$, depending only on the
parameters listed in Remark~\textup{\ref{rem:constants}}, such that
\begin{equation}
\label{eq:main_TV}
 D_{TV}(\mu_{data},\hat{\mu})
 \le C\sqrt{D_{KL}(\nu\,\|\,\mu_T)+R_{\mathrm{score}}+R_{\mathrm{disc}}},
\end{equation}
where, with $\bar J$ as in~\eqref{eq:fisher_uniform} and
$L_{\mathrm{rev}}$ as in~\eqref{eq:Lscore},
\begin{align}
\label{eq:main_Rscore}
 R_{\mathrm{score}}
 &\le C\,\Lambda_0\,T L_{\mathrm{rev}},\\
\label{eq:main_Rdisc}
 R_{\mathrm{disc}}
 &\le C\Delta_{\max}\Big\{
 \Lambda_0\big[J_1+(1+M_A+L_b)T\bar J+TG_b^2\big]
 +\lambda_0^{-1}L_b^2(1+\bar J)T\Big\}
 \notag\\
 &\quad+C\Lambda_0\big[(M_A+L_b)^2\bar J+G_b^2\big]T\Delta_{\max}^2
 +C\lambda_0^{-1}L_{b,t}^2T\Delta_{\max}^2.
\end{align}
\end{cor}
\begin{proof}
Since $\hat\mu_{\mu_T}=\mu_T\widehat K$ and $\hat\mu=\nu\widehat K$, the
triangle inequality and contraction of total variation under a Markov kernel give
\begin{equation*}
 D_{TV}(\mu_{data},\hat{\mu})
 \le D_{TV}(\mu_{data},\hat\mu_{\mu_T})
   +D_{TV}(\mu_T\widehat K,\nu\widehat K)
 \le D_{TV}(\mu_{data},\hat\mu_{\mu_T})+D_{TV}(\mu_T,\nu).
\end{equation*}
Theorem~\ref{thm:disc_score} bounds the first term by
$\tfrac12\sqrt{R_{\mathrm{score}}+R_{\mathrm{disc}}}$, while Pinsker bounds
the second by $\sqrt{\tfrac12D_{KL}(\nu\,\|\,\mu_T)}$. Squaring and using
$(x+y)^2\le 2x^2+2y^2$, we get
\[
 D_{TV}(\mu_{data},\hat{\mu})^2
 \le 2\cdot\tfrac12 D_{KL}(\nu\,\|\,\mu_T)
   + 2\cdot\tfrac14\big(R_{\mathrm{score}}+R_{\mathrm{disc}}\big)
 = D_{KL}(\nu\,\|\,\mu_T) + \tfrac12\big(R_{\mathrm{score}}+R_{\mathrm{disc}}\big),
\]
which is~\eqref{eq:main_TV}. The bounds~\eqref{eq:main_Rscore}--\eqref{eq:main_Rdisc}
are~\eqref{eq:R_score}--\eqref{eq:R_disc} of Theorem~\ref{thm:disc_score}.
\end{proof}

\begin{rem}[Limitations and possible extensions]
\label{cor:wass}
Corollary~\ref{cor:main} does not by itself yield a Wasserstein bound.
Talagrand controls $W_2(\hat\mu,\mu_{data})$ through
$D_{KL}(\hat\mu\,\|\,\mu_{data})$, while Theorem~\ref{thm:disc_score} controls
the opposite path-space entropy.  Reversing the entropy direction requires
sampler-path expectations and spatial derivatives of the score error.
Two natural extensions also remain.  State-dependent $\sigma$ introduces
$\nabla a$ terms and removes the Gaussian reference.  An unbounded nonlinear
drift requires a random inverse first-variation flow and new reverse-score
BSDE estimates.  Even for bounded nonlinear $b$, deriving
Assumption~\textup{\ref{lem:kernel}(iii)} and the endpoint condition in
Assumption~\ref{ass:score_closure} from the coefficient conditions remains open.
\end{rem}

\subsection{A first-order refinement under smooth-density regularity}
\label{sec:first_order}

The square root in Theorem~\ref{thm:disc_score} is caused by applying
Pinsker's inequality to a path-space relative entropy.  It is robust under
the first-order assumptions used above, but it does not exploit cancellation
in the terminal density.  The following theorem records the classical
smooth-density regime in which that cancellation restores order one.

For an integer $k\ge0$ and $m\ge0$, write
\[
 \|q\|_{W^{k,1}_m}
 :=\sum_{|\alpha|\le k}\int_{\mathbb R^d}
       (1+|x|^m)|\partial^\alpha q(x)|\,dx.
\]
We write $W^{k,1}:=W^{k,1}(\mathbb R^d)$ for the usual unweighted Sobolev
space.
For $r\ge0$, let $C^3_{\mathrm{pol},r}$ denote the functions $u(t,x)$ for
which $u(t,\cdot)\in C^3(\mathbb R^d)$ and
\[
 \sup_{0\le t\le T}|\partial_x^\alpha u(t,x)|
 \le C(1+|x|^r),\qquad |\alpha|\le3,\quad x\in\mathbb R^d,
\]
with a common constant $C$.  No time derivative of $u$ is included in this
notation.

Put
\[
 s(t,x):=\nabla\log p_{T-t}(x),\qquad q_t:=p_{T-t},
\]
and let
\[
 F(t,x):=-A(T-t)x-b(T-t,x)+a(T-t)s(t,x)
\]
be the exact reverse drift.  For a density $v$, set
\[
 \mathcal L_t^*v
 :=-\operatorname{div}(F(t,\cdot)v)
   +\tfrac12\operatorname{tr}\!\big(a(T-t)\nabla^2v\big).
\]

\begin{assum}[Smooth-density condition]
\label{ass:smooth_density}
For some $r\ge0$, $A,a\in W^{1,\infty}([0,T])$, and the functions
$b(T-t,x)$ and $s(t,x)$ belong to $C^3_{\mathrm{pol},r}$ and are globally
Lipschitz in $x$, uniformly in $t$.  For an integer $m\ge2(r+1)$,
\begin{equation}
\label{eq:strong_density_reg}
 q\in C^{1,1}([0,T];L^1(\mathbb R^d)),\qquad
 \sup_{0\le t\le T}\|q_t\|_{W^{4,1}_m}<\infty,
\end{equation}
where $C^{1,1}$ means that $t\mapsto\partial_tq_t$ is Lipschitz in $L^1$.
\end{assum}

Under this regularity, the Fokker--Planck equation for the exact reverse
process holds in $L^1(\mathbb R^d)$, so
$\partial_tq_t=\mathcal L_t^*q_t$.

Related density expansions for the Euler scheme were established by Bally
and Talay~\cite{bal-tal:1996,bal-tal:1996b}; the proof here works directly
with the Gaussian kernel of~\eqref{eq:discrete}.

\begin{samepage}
\begin{thm}[First-order total-variation discretization rate]
\label{thm:first_order_tv}
Assume the hypotheses of Theorem~\textup{\ref{thm:disc_score}} and
Assumption~\textup{\ref{ass:smooth_density}}.  Use the exact score at the grid points
and initialize the numerical scheme from $\mu_T$.
Then there exist positive constants $C_{\mathrm{reg}}$ and $\Delta_0$,
independent of the partition, such that every partition with
$\Delta_{\max}\le\Delta_0$ satisfies
\begin{equation*}
 D_{TV}(\mu_{data},\hat\mu_{\mu_T})
 \le C_{\mathrm{reg}}\sum_{i=0}^{n-1}\Delta_i^2
 \le C_{\mathrm{reg}}T\Delta_{\max}.
\end{equation*}
\end{thm}
\end{samepage}

\begin{proof}
Let $k_{t,h}(x,y)$ be the transition density of one step of
\eqref{eq:discrete} from reverse time $t$ with step $h$, and define
\[
 (Q_{t,h}^*v)(y):=\int_{\mathbb R^d}k_{t,h}(x,y)v(x)\,dx.
\]
We first prove that there are positive constants $C$ and $h_0$, independent
of $t$ and $h$, such that
\begin{equation}
\label{eq:local_L1_consistency}
 \|q_{t+h}-Q_{t,h}^*q_t\|_{L^1}\le Ch^2
\end{equation}
whenever $0<h\le h_0$ and $t+h\le T$.
For the exact-score scheme, the one-step numerical state $\hat X_{t+h}$,
conditional on $\hat X_t=x$, has the Gaussian law
$N(M_{t,h}(x),\Gamma_{t,h})$, where
\[
 M_{t,h}(x)=\Psi(t+h,t)x
 +\int_t^{t+h}\!\Psi(t+h,u)a(T-u)\,du\,s(t,x)
 -\int_t^{t+h}\!\Psi(t+h,u)\,du\,b(T-t,x)
\]
and
\[
 \Gamma_{t,h}=\int_t^{t+h}\Psi(t+h,u)a(T-u)
                  \Psi(t+h,u)^{\mathsf T}\,du.
\]
The time-Lipschitz regularity of $A$ and $a$ gives
\begin{equation}
\label{eq:mean_cov_expansion}
 M_{t,h}(x)=x+hF(t,x)+h^2R_{t,h}(x),\qquad
 \Gamma_{t,h}=h\,a(T-t)+h^2S_{t,h}.
\end{equation}
Here $R_{t,h}:\mathbb R^d\to\mathbb R^d$ and
$S_{t,h}\in\mathbb R^{d\times d}$ are the normalized remainder terms in the
mean and covariance expansions, respectively.
The assumed spatial-derivative bounds show that the derivatives of
$R_{t,h}$ through order three have polynomial growth of a fixed order,
uniformly for $h\le h_0$, while its first derivative is bounded.  The
matrices $S_{t,h}$ are uniformly bounded.  Decreasing $h_0$ if necessary,
$M_{t,h}$ and all the interpolation maps
$M_{t,h}^{\tau}(x):=x+\frac{\tau}{h}(M_{t,h}(x)-x)$, $0\le\tau\le h$, are
global $C^1$ diffeomorphisms with Jacobians bounded above and away from
zero.

We first use an elementary pushforward estimate.  Let
$q\in W^{2,1}_m(\mathbb R^d)$ be a probability density.  If
$M_h(x)=x+hf(x)+h^2R_h(x)$ has the preceding bounds and $\rho_h$ is determined by
$(M_h)_\#(q(x)\,dx)=\rho_h(y)\,dy$, interpolate over one time step by
\[
 M_h^\tau(x):=x+\tau\bigl(f(x)+hR_h(x)\bigr),
 \qquad 0\le\tau\le h.
\]
This freezes the velocity $f+hR_h$ at the initial point, and $M_h^h=M_h$.
Change of variables along $M_h^\tau$, followed by the fundamental theorem of
calculus in~$\tau$, gives
\begin{equation}
\label{eq:pushforward_second_order}
 \|\rho_h-q+h\,\operatorname{div}(fq)\|_{L^1}
 \le Ch^2\|q\|_{W^{2,1}_m}.
\end{equation}
Here are details of these estimates.  Put
$\beta_h=f+hR_h$, let $\rho_h^\tau$ be determined by
$(M_h^\tau)_\#(q(x)\,dx)=\rho_h^\tau(y)\,dy$, and set
$v_h^\tau=\beta_h\circ(M_h^\tau)^{-1}$.  The change-of-variables
identity implies, in distributions,
\[
 \partial_\tau\rho_h^\tau
   =-\operatorname{div}(v_h^\tau\rho_h^\tau),
 \qquad
 \partial_\tau v_h^\tau
   =-[D\beta_h\circ(M_h^\tau)^{-1}]
     [DM_h^\tau\circ(M_h^\tau)^{-1}]^{-1}v_h^\tau
   =-(Dv_h^\tau)v_h^\tau.
\]
Consequently
\[
 \partial_\tau^2\rho_h^\tau
 =\operatorname{div}\!\left(
   (Dv_h^\tau)v_h^\tau\rho_h^\tau\right)
  +\operatorname{div}\!\left(
   v_h^\tau\operatorname{div}(v_h^\tau\rho_h^\tau)\right).
\]
The near-identity bounds on $M_h^\tau$ and its inverse therefore give
\[
 \sup_{0\le\tau\le h}\|\partial_\tau^2\rho_h^\tau\|_{L^1}
 \le C\|q\|_{W^{2,1}_m},
 \qquad
 \partial_\tau\rho_h^\tau|_{\tau=0}
 =-\operatorname{div}(fq)-h\operatorname{div}(R_hq).
\]
Taylor's formula in $\tau$ gives
\[
 \rho_h-q+h\operatorname{div}(fq)
 =-h^2\operatorname{div}(R_hq)
  +\int_0^h(h-\tau)\partial_\tau^2\rho_h^\tau\,d\tau,
\]
which proves
\eqref{eq:pushforward_second_order}.  The resulting coefficients contain
products of at most two polynomially growing derivatives of $f$ and $R_h$;
$m\ge2(r+1)$ controls these products and also the linear growth of $f$ and
$R_h$.

We next use two estimates for Gaussian convolution.  If $G_\Gamma$ denotes
the centered Gaussian density with covariance $\Gamma$, Taylor's formula
for translations, the vanishing of its odd moments, and Fubini's theorem
give
\begin{equation}
\label{eq:gaussian_L1_expansion}
 \|G_\Gamma*v-v\|_{L^1}
 \le C\|\Gamma\|_{\mathrm{op}}\|v\|_{W^{2,1}},\qquad v\in W^{2,1},
\end{equation}
and
\begin{equation}
\label{eq:gaussian_L1_second_expansion}
 \|G_\Gamma*v-v-\tfrac12\tr(\Gamma\nabla^2v)\|_{L^1}
 \le C\|\Gamma\|_{\mathrm{op}}^2\|v\|_{W^{4,1}},\qquad v\in W^{4,1}.
\end{equation}
These inequalities follow first for smooth compactly supported $v$ from
the integral Taylor remainders and the second and fourth Gaussian moments,
and then by density.

Apply these estimates with $f=F(t,\cdot)$ and $q=q_t$.  Let $\rho_{t,h}$
be the density determined by
$(M_{t,h})_\#(q_t(x)\,dx)=\rho_{t,h}(y)\,dy$.  Put
\[
 g_t:=\operatorname{div}(F(t,\cdot)q_t),\qquad
 e_{t,h}:=\rho_{t,h}-q_t+h g_t.
\]
Then~\eqref{eq:pushforward_second_order} gives $\|e_{t,h}\|_{L^1}\le Ch^2$,
while the product rule and the $C^3_{\mathrm{pol},r}$ bounds give
$\|g_t\|_{W^{2,1}}\le C\|q_t\|_{W^{3,1}_m}$.  Write
$\Gamma=\Gamma_{t,h}$.  Since $Q_{t,h}^*q_t=G_\Gamma*\rho_{t,h}$, we have
the exact decomposition
\begin{align*}
 Q_{t,h}^*q_t-q_t-h\mathcal L_t^*q_t
 ={}&G_\Gamma*e_{t,h}-h(G_\Gamma*g_t-g_t)+G_\Gamma*q_t-q_t\\
 &-\tfrac12\tr(\Gamma\nabla^2q_t)
  +\tfrac12\tr\!\bigl((\Gamma-h a(T-t))\nabla^2q_t\bigr).
\end{align*}
Gaussian convolution is an $L^1$ contraction.  Therefore
\eqref{eq:mean_cov_expansion},
\eqref{eq:gaussian_L1_expansion}, and
\eqref{eq:gaussian_L1_second_expansion} yield
\begin{equation}
\label{eq:numerical_local_expansion}
 \big\|Q_{t,h}^*q_t-q_t-h\mathcal L_t^*q_t\big\|_{L^1}
 \le Ch^2.
\end{equation}

To complete the local estimate, the exact density curve and its
Fokker--Planck equation give the
Banach-space identity
\[
 q_{t+h}=q_t+h\mathcal L_t^*q_t
       +\int_0^h(\partial_tq_{t+u}-\partial_tq_t)\,du.
\]
Its remainder is bounded by $Ch^2$ in $L^1$ by the $C^{1,1}$ condition in
\eqref{eq:strong_density_reg}.  Combining this identity with
\eqref{eq:numerical_local_expansion} proves
\eqref{eq:local_L1_consistency}.

We now pass from the local estimate to the terminal law.  Set
$\Delta_0:=h_0$.
Let $P_i^*$ and $Q_i^*$ be the exact and numerical density operators from
$t_i$ to $t_{i+1}$.  Estimate~\eqref{eq:local_L1_consistency}, applied with
$t=t_i$ and $h=\Delta_i$, gives
\begin{equation*}
 \|(P_i^*-Q_i^*)q_{t_i}\|_{L^1}
 =\|q_{t_{i+1}}-Q_i^*q_{t_i}\|_{L^1}
 \le C_{\mathrm{reg}}\Delta_i^2.
\end{equation*}
The exact hybrid-product identity is
\[
 P_{n-1}^*\cdots P_0^*-Q_{n-1}^*\cdots Q_0^*
 =\sum_{i=0}^{n-1}Q_{n-1}^*\cdots Q_{i+1}^*
   (P_i^*-Q_i^*)P_{i-1}^*\cdots P_0^*.
\]
Every Markov density operator is an $L^1$ contraction, and
$P_{i-1}^*\cdots P_0^*p_T=q_{t_i}$.  Hence
\[
 \|P_{n-1}^*\cdots P_0^*p_T
      -Q_{n-1}^*\cdots Q_0^*p_T\|_{L^1}
 \le\sum_{i=0}^{n-1}\|(P_i^*-Q_i^*)q_{t_i}\|_{L^1}
 \le C_{\mathrm{reg}}\sum_{i=0}^{n-1}\Delta_i^2.
\]
The exact terminal density is $p_0$, and total variation is one half of the
$L^1$ distance.  Finally,
$\sum_i\Delta_i^2\le T\Delta_{\max}$.
\end{proof}

\begin{ex}[Symmetric Gaussian mixture with a constant VP schedule]
\label{ex:first_order_mixture}
Let $d=1$, $T=1$, $b=0$, $A=-a/2$, and let the diffusion matrix be the
positive constant $a$, where $a>0$.  If
\[
 \mu_{data}=\tfrac12N(-\mu,\sigma_0^2)
             +\tfrac12N(\mu,\sigma_0^2),
 \qquad \mu>0,\quad \sigma_0>0,
\]
then at forward time $u$ the density is the equal mixture with component
means $\pm c_u$ and common variance $v_u$, where
\[
 c_u=\mu e^{-au/2},\qquad
 v_u=1+(\sigma_0^2-1)e^{-au}.
\]
In particular, $v_u\ge\min\{1,\sigma_0^2\}>0$, and its score is
\[
 \partial_x\log p_u(x)
 =-\frac{x}{v_u}+\frac{c_u}{v_u}
    \tanh\!\left(\frac{c_ux}{v_u}\right).
\]
For fixed $(a,\mu,\sigma_0)$ this score is globally Lipschitz in $x$,
uniformly in $u\in[0,1]$.  Its spatial derivatives through order three have
at most quadratic growth, uniformly in time.  Moreover, $q_t=p_{1-t}$ is a
smooth Gaussian mixture, belongs to $C^{1,1}([0,1];L^1)$, and, for every
finite $m$,
\[
 \sup_{0\le t\le1}\|q_t\|_{W^{4,1}_m}<\infty.
\]
Thus Assumption~\textup{\ref{ass:smooth_density}} holds, for example with $r=2$
and $m\ge6$.  Theorem~\textup{\ref{thm:first_order_tv}} therefore gives, for
the exact-score sampler,
\[
 D_{TV}(\mu_{data},\hat\mu_{\mu_1})=O(\Delta_{\max}).
\]
\end{ex}

\begin{rem}[Why the minimal-regularity proof gives order one-half]
\label{rem:half_rate}
The martingale term $\int Z_r\,dW_r^*$ in the reverse-score increment gives
$\mathbb E\int_0^T|Y_t-Y_{\tau_n(t)}|^2dt=O(\Delta_{\max})$.
The stationary Ornstein--Uhlenbeck process already exhibits this behavior,
with $Z$ a nonzero constant matrix.  The resulting path-space entropy
estimate in~\eqref{eq:girKL} is $O(\Delta_{\max})$, which Pinsker converts
into $O(\Delta_{\max}^{1/2})$.  Theorem~\textup{\ref{thm:first_order_tv}}
obtains first-order terminal accuracy through a different mechanism, namely
cancellation after projection to the terminal density.  Thus the
low-regularity energy argument provides a half-order terminal-law bound,
while establishing its sharpness would require a separate lower-bound
analysis.
\end{rem}

\section{Specializations}\label{sec:specialize}

\subsection{Variance-preserving exponential integrator}\label{sec:ddpm}

Set $A(t)=-\tfrac12\beta_tI_d$, $b=0$, and
$\sigma(t)=\sqrt{\beta_t}I_d$ on $[0,1]$, where the piecewise-constant
schedule $\beta=g'$ is defined in Remark~\ref{rem:ddpm_exact}, and take
$\nu=N(0,I_d)$.  Remark~\ref{rem:DDPM_fisher} gives
$\bar J\le C(J_1+d)$.

The next result specializes Corollary~\ref{cor:main} to this variance-preserving schedule.

\begin{cor}[Variance-preserving bound]
\label{cor:ddpm}
Under \textup{(D)}, suppose the schedule above satisfies
$0<\beta_{\min}\le\beta_t\le\beta_{\max}<\infty$ and
$L_{\mathrm{rev}}<\infty$.  For the uniform grid $t_i=i/n$, if
$n\ge\beta_{\max}$, then the sampler initialized at $\nu=N(0,I_d)$ satisfies
\begin{equation}
\label{eq:ddpm_bound}
 D_{TV}(\mu_{data},\hat{\mu})
 \le C\sqrt{D_{KL}(N(0,I_d)\,\|\,\mu_T)
   +\beta_{\max} L_{\mathrm{rev}}
   +\frac{\beta_{\max}}{n}
      \big[J_1+(1+\beta_{\max})(J_1+d)\big]},
\end{equation}
where $\alpha_{\min}=\min_i\alpha_i$ and
$\beta_{\max}=-n\log\alpha_{\min}$.  Moreover,
\begin{equation}
\label{eq:vp_init_explicit}
 D_{KL}(N(0,I_d)\,\|\,\mu_T)
 \le \frac{\bar\alpha_n}{2(1-\bar\alpha_n)}
       (d+M_2).
\end{equation}
\end{cor}
\begin{proof}
For $b=0$, the kernel estimates hold with $c_*=1$ and $\vartheta=0$, while
Proposition~\ref{prop:linear_closure} gives the reverse-score identity.
Apply Corollary~\ref{cor:main} with
$\Lambda_0=\beta_{\max}$, $M_A=\beta_{\max}/2$,
$\Delta_{\max}=1/n$, and $\bar J\le C(J_1+d)$.  Then
$R_{\mathrm{score}}\le C\beta_{\max}L_{\mathrm{rev}}$ and
\[
 R_{\mathrm{disc}}\le
 C\frac{\beta_{\max}}n
   [J_1+(1+\beta_{\max})(J_1+d)]
 +C\frac{\beta_{\max}^3}{n^2}(J_1+d).
\]
The second term is absorbed by the first when $n\ge\beta_{\max}$.
Equation~\eqref{eq:ddpm_bound} now follows from~\eqref{eq:main_TV}, and
\eqref{eq:vp_init_explicit} follows from Remark~\ref{rem:initial_ddpm}.
\end{proof}

\begin{rem}[DDPM as a specialization]
\label{rem:comparison_nak}
For fixed $\beta_{\max}$, Corollary~\ref{cor:ddpm} gives the standard
Girsanov--Pinsker rate $O(n^{-1/2})$.  Its benefit is that the two-regime
Fisher bound avoids the factor $(\bar\alpha_n)^{-1}$ produced by the
pointwise estimate $\mathbb E|\nabla\log p_t|^2\le CJ_1/m_{0,t}^2$.
An earlier analysis incurred $e^{c(\bar\alpha_n)^{-1}}$~\cite{nak:2024}.
Polynomial dependence on $\beta_{\max}$ remains.  See
Table~\ref{tab:comparison}.
\end{rem}

\subsection{Verified nonlinear examples}
\label{sec:doob_perturbation}

The kernel hypothesis can be checked beyond the linear case.  Consider the
Brownian reference $A=0$, $a=I_d$ and let
\[
 h(t,x):=\mathbb E[\psi(x+W_{T-t})],\qquad
 b(t,x):=\nabla\log h(t,x),
\]
where $\psi\in C_b^4(\mathbb R^d)$ and
$0<c_-\le\psi\le c_+<\infty$.

\begin{prop}[A nonlinear class satisfying the kernel hypothesis]
\label{prop:doob_kernel}
The preceding drift satisfies \textup{(F)} and Assumption~\textup{\ref{lem:kernel}}.  Its
transition density is
\begin{equation}
\label{eq:doob_kernel}
 p(t,x,s,y)=\varphi_{(s-t)I_d}(y-x)\frac{h(s,y)}{h(t,x)},
 \qquad 0\le t<s\le T,
\end{equation}
and Assumption~\textup{\ref{lem:kernel}(iii)} holds with
\begin{equation}
\label{eq:doob_theta}
 \vartheta(t,x,s,y)=\nabla\log h(s,y)-\nabla\log h(t,x).
\end{equation}
\end{prop}
\begin{proof}
For $|\alpha|\le4$, the heat-semigroup representation gives
$\partial_x^\alpha h(t,\cdot)=P_{T-t}(\partial^\alpha\psi)$, so these
derivatives are uniformly bounded.  Since $c_-\le h\le c_+$, the quotient
formulas for $b=\nabla h/h$ show that $b$, $\nabla b$, and
$\nabla\operatorname{div}b$ are uniformly bounded.  Moreover,
$\partial_t h=-\tfrac12\Delta h$ implies that $\partial_t b$ is uniformly
bounded.  Thus $b$ is globally Lipschitz in space and uniformly Lipschitz in
time.  Together with $A=0$ and $a=I_d$, this verifies \textup{(F)}.
Formula~\eqref{eq:doob_kernel} and the same bounds give the two-sided
Gaussian estimate, while differentiation in $y$ gives the log-gradient
estimate.  Moreover,
$\nabla_y\log\varphi_{(s-t)I_d}(y-x)
=-\nabla_x\log\varphi_{(s-t)I_d}(y-x)$, so differentiating
\eqref{eq:doob_kernel} in both endpoints yields~\eqref{eq:doob_theta}.
It is uniformly bounded because $\nabla\log h$ is bounded.
\end{proof}

For example, $\psi(x)=1+\varepsilon\cos(k^{\mathsf T}x)$, $k\in\mathbb R^d$,
$0<|\varepsilon|<1$,
gives an explicit bounded nonlinear member, since
$h(t,x)=1+\varepsilon e^{-|k|^2(T-t)/2}\cos(k^{\mathsf T}x)$.

The next example retains the nonlinear term in the exact reverse drift and
falls under both Corollary~\ref{cor:main} and
Theorem~\ref{thm:first_order_tv}.

\begin{prop}[A complete nonlinear example with noncancelling drift]
\label{prop:noncancelling}
Let $d=1$, $A=-\kappa$, $a=1$, and
\[
 b(x)=\frac{\varepsilon x}{1+x^2},\qquad
 p_{data}(x)=Z^{-1}e^{-\kappa x^2}(1+x^2)^\varepsilon,
 \quad \kappa>0,\ \varepsilon\in\mathbb R,
\]
where $Z$ is the normalizing constant.
Then \textup{(D)}, \textup{(F)}, Assumptions~\textup{\ref{ass:score_closure}} and
\textup{\ref{lem:kernel}}
hold.  The forward law is stationary, its score is
$-2\kappa x+2b(x)$, and the exact reverse drift is $-\kappa x+b(x)$.
The exact-score scheme initialized from $\mu_T$ satisfies
$D_{TV}(\mu_{data},\hat\mu_{\mu_T})=O(\Delta_{\max})$.
\end{prop}
\begin{proof}
The displayed density is invariant because its score is twice the forward
drift.  Its Gaussian tail and smooth score verify (D).  Since
$b\in C_b^\infty$, the density belongs to $W^{4,1}_m$ for every $m\ge0$, and
its score is globally Lipschitz.  Thus (F),
Assumption~\textup{\ref{ass:score_closure}(i)}, and
Assumption~\textup{\ref{ass:smooth_density}} hold with $r=1$ and any $m\ge4$.
Stationarity and path continuity give
Assumption~\textup{\ref{ass:score_closure}(ii)}.

It remains to check the kernel estimates.  For $h\in(0,T]$, put
$m_h=e^{-\kappa h}$ and $v_h=(1-m_h^2)/(2\kappa)$.  Since the coefficients
are time homogeneous, $p(t,x,s,y)=p_h(x,y)$,
$\Phi(s,t)=m_h$, and $V(t,s)=v_h$ whenever $h=s-t$.  Hence it suffices to
verify the estimates uniformly for $h\in(0,T]$.  Let
$q_h(x,y)=\varphi_{v_h}(y-m_hx)$ be the OU kernel.  With
$B(x)=\frac\varepsilon2\log(1+x^2)$, we have $B'=b$.  Under the OU
law $\mathbb P_x^{\mathrm{OU}}$, let
$dZ_u=-\kappa Z_u\,du+dW_u$, $Z_0=x$, and put
\[
 M_u:=\exp\!\left(
 \int_0^u b(Z_v)\,dW_v-\frac12\int_0^u b(Z_v)^2\,dv\right),
 \qquad 0\le u\le h.
\]
Since $b$ is bounded, Novikov's condition holds and $M$ is a martingale
with $\mathbb E_x^{\mathrm{OU}}M_h=1$.  Under the probability measure
with density $M_h$ relative to $\mathbb P_x^{\mathrm{OU}}$ on $[0,h]$,
Girsanov's theorem gives
$dZ_u=(-\kappa Z_u+b(Z_u))\,du+d\widetilde W_u$, where
$\widetilde W_u=W_u-\int_0^u b(Z_v)\,dv$ is Brownian motion under the new
measure.  Uniqueness in law identifies the transition density with $p_h$.
It\^o's formula gives
\[
 B(Z_h)-B(x)=\int_0^h b(Z_u)\,dW_u
 +\int_0^h\left(-\kappa Z_ub(Z_u)+\frac12b'(Z_u)\right)du.
\]
Hence, with
$\mathcal V(x):=-\kappa xb(x)+\frac12b(x)^2+\frac12b'(x)$,
\[
 M_h=\exp\!\left(B(Z_h)-B(x)-\int_0^h\mathcal V(Z_u)\,du\right).
\]
Conditioning on $Z_h=y$ therefore yields
\[
 \frac{p_h(x,y)}{q_h(x,y)}
 =e^{B(y)-B(x)}r_h(x,y),\qquad
 r_h(x,y):=\mathbb E^{\mathrm{OU}}_{x,y}
 \exp\!\left(-\int_0^h\mathcal V(Z_u)\,du\right),
\]
where $\mathbb E_{x,y}^{\mathrm{OU}}$ denotes expectation under the
Ornstein--Uhlenbeck bridge from $x$ at time $0$ to $y$ at time $h$.
Since
$\mathcal V,\mathcal V'$ are bounded
and the OU bridge is affine in its endpoints with endpoint-independent
covariance, differentiation under the expectation gives
$e^{-Ch}\le r_h\le e^{Ch}$ and
$|\partial_x\log r_h|+|\partial_y\log r_h|\le Ch$.
Moreover,
\[
 |B(m_hx)-B(x)|\le\kappa|\varepsilon|h,
 \qquad |B(y)-B(m_hx)|\le\|b\|_\infty|y-m_hx|.
\]
Absorbing the second bound into
fixed dilations of the OU Gaussian proves Assumption~\textup{\ref{lem:kernel}(i)}.

Differentiation with respect to the terminal variable $y$ gives
\[
 \partial_y\log p_h=-\frac{y-m_hx}{v_h}+b(y)+\partial_y\log r_h.
\]
Since $b$ is bounded, $|\partial_y\log r_h|\le Ch$, and $v_h\le T$, we obtain
\[
 |\partial_y\log p_h(x,y)|
 \le \frac{|y-m_hx|}{v_h}+\frac{C}{\sqrt{v_h}},
\]
which proves (ii).

To verify (iii), we combine the derivatives in the initial and terminal
variables.  The OU terms cancel, giving
\[
 \partial_y\log p_h+m_h^{-1}\partial_x\log p_h
 =b(y)-m_h^{-1}b(x)
  +\partial_y\log r_h+m_h^{-1}\partial_x\log r_h.
\]
With $h=s-t$, the right-hand side defines $\vartheta(t,x,s,y)$.  It is Borel
measurable and uniformly bounded for $0\le t<s\le T$, which proves (iii).
Thus Assumption~\textup{\ref{lem:kernel}} holds.  For the exact score,
$L_{\mathrm{rev}}=0$.  Hence Corollary~\textup{\ref{cor:main}} and
Theorem~\textup{\ref{thm:first_order_tv}} apply.
\end{proof}

\subsection{Variance-exploding schedule}

For $A=b=0$ and $\sigma(t)=\sqrt{\gamma(t)}I_d$, put
$v_t:=\int_0^t\gamma(s)ds$.  Then
$X_t=\mathbf x_0+\sqrt{v_t}Z$, and the two score representations give
\begin{equation}
\label{eq:ve_fisher}
 \mathbb E|\nabla\log p_t(X_t)|^2
 \le \min\{J_1,d/v_t\},
 \qquad \bar J\le J_1.
\end{equation}

The next result specializes Corollary~\ref{cor:main} to this variance-exploding schedule.

\begin{cor}[Variance-exploding bound]
\label{cor:ve}
Under \textup{(D)}, let $\gamma$ be piecewise continuous with finitely many
discontinuities and
$0<\gamma_{\min}\le\gamma(t)\le\gamma_{\max}<\infty$, and
assume $L_{\mathrm{rev}}<\infty$.  Let $\hat\mu$ be the output of
\eqref{eq:discrete} from the Gaussian initialization
$\nu=N(0,v_TI_d)$. Then, on any grid,
\begin{equation}
\label{eq:ve_bound}
 D_{TV}(\mu_{data},\hat\mu)
 \le C\sqrt{
   \frac{M_2}{v_T}
   +\gamma_{\max}T L_{\mathrm{rev}}
   +\gamma_{\max}\Delta_{\max}J_1(1+T)}.
\end{equation}
\end{cor}
\begin{proof}
The Gaussian kernel and Proposition~\ref{prop:linear_closure} give the
reverse-score identity.  Jensen's inequality applied to the data and
Gaussian score representations gives~\eqref{eq:ve_fisher}.
Corollary~\ref{cor:main}, with
$M_A=L_b=L_{b,t}=G_b=0$, $\Lambda_0=\gamma_{\max}$, and
$\bar J\le J_1$, gives the score and discretization terms in
\eqref{eq:ve_bound}.
For the initialization, if $\phi_{v_T}$ is the density of $N(0,v_TI_d)$,
then
\[
 \frac{p_T(y)}{\phi_{v_T}(y)}
 =\mathbb E\exp\!\left(
    \frac{y\cdot\mathbf x_0}{v_T}
    -\frac{|\mathbf x_0|^2}{2v_T}\right).
\]
Jensen's inequality and $\mathbb E_{Y\sim N(0,v_TI_d)}Y=0$ yield
\[
 D_{KL}(N(0,v_TI_d)\,\|\,\mu_T)
 \le\frac{M_2}{2v_T}.
\]
Substitution into~\eqref{eq:main_TV} proves the claim.
\end{proof}

\subsection{A numerical illustration}\label{sec:numerics}

We test the three error components in Corollary~\ref{cor:main} and the
noncancelling example of Proposition~\ref{prop:noncancelling}.  For the VP
tests, $d=1$, $\mu_{data}=\tfrac12N(-3,\tfrac14)+\tfrac12N(3,\tfrac14)$,
and $-\log\alpha_i=a/n$, so the marginals and scores are explicit.  The
score experiment uses $5\times10^5$ paired samples to form $200$-bin
histograms on $[-8,8]$ and reports their TV distance.  The deterministic tests
use conservative linear deposition and fast Fourier transform (FFT) Gaussian
convolution on $24001$ nodes over $[-18,18]$ and compute
$\frac12\|q-p_0\|_{L^1}$.  The code and numerical values used to generate
Figure~\ref{fig:numerics} are included in the ancillary files accompanying
the arXiv version.

\begin{figure}[H]
\centering
\includegraphics[width=0.96\textwidth]{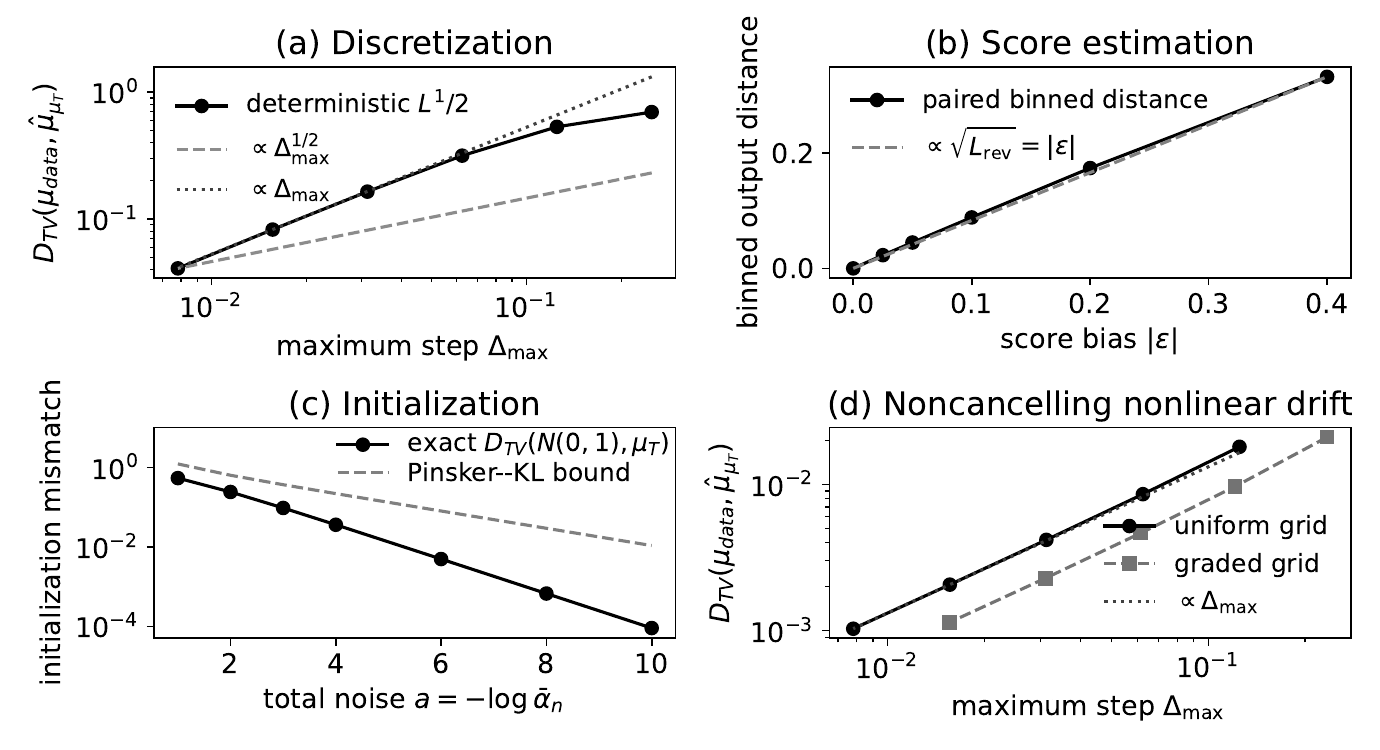}
\caption{Numerical checks of (a) exact-score discretization, (b) binned score
error, (c) initialization and its Pinsker--KL bound, and (d) nonlinear drift on
two grids.}
\label{fig:numerics}
\end{figure}

\emph{Discretization} (Figure~\ref{fig:numerics}(a)).  With exact $p_T$
initialization and score at $a=10$, the $L^1/2$ error decreases from $0.697$
to $0.0408$ for $n=4,8,16,32,64,128$.  The last four levels give empirical order
$0.985$, consistent with the first-order convergence established in
Theorem~\textup{\ref{thm:first_order_tv}}.
Example~\ref{ex:first_order_mixture} verifies the regularity assumptions
for this Gaussian mixture.

\emph{Score estimation} (Figure~\ref{fig:numerics}(b)).  At $a=6$ and
$n=128$, replace $s$ by $s_\varepsilon=s+\varepsilon$.  The paired histogram
distance increases from $0.023$ to $0.331$ over
$\varepsilon=0.025,0.05,0.1,0.2,0.4$.  Its near-linear dependence on
$|\varepsilon|=\sqrt{L_{\mathrm{rev}}}$ is consistent with~\eqref{eq:main_Rscore}.
This controlled perturbation is not a trained-network experiment.

\emph{Initialization} (Figure~\ref{fig:numerics}(c)).  Direct quadrature shows
that $D_{TV}(N(0,1),\mu_T)$ decreases from $0.537$ to $9.1\times10^{-5}$ for
$a=1,2,3,4,6,8,10$.  The explicit Pinsker--KL bound is conservative, and its
propagated effect can only decrease by contraction.

\emph{Nonlinear drift} (Figure~\ref{fig:numerics}(d)).  With $\kappa=1$ and
$\varepsilon=0.8$ in Proposition~\ref{prop:noncancelling}, the uniform and
$t_i=1-(1-i/n)^2$ grids started from the exact stationary law give errors
$0.0181$ to $0.00103$ and $0.0210$ to $0.00114$, respectively, for
$n=8,16,32,64,128$.  Their last-three-level orders against $\Delta_{\max}$ are
$1.010$ and $1.024$.  The experiment supports first-order convergence with a
frozen nonvanishing $b$ term, not grid optimality.

\section{Conclusion}\label{sec:conclusion}

We prove half-order bounds for nonautonomous linear-plus-bounded diffusions and
first-order convergence for exact scores under smooth densities.  A stationary
nonlinear example satisfies all assumptions.  Extending the first-order
time-discretization bound to learned scores is a direction for future work.

\subsection*{Acknowledgements}

This study is supported by JSPS KAKENHI Grant Number JP24K06861.

The author used OpenAI Codex to suggest clearer and shorter wording, check
notation, definitions, citations, and bibliographic information, assist in
checking and refining mathematical arguments and numerical code, and perform
LaTeX and compilation checks. The author reviewed and approved all claims,
proofs, references, code, and final wording and assumes responsibility for
all content.


\begin{thebibliography}{99}

\bibitem{aronson:1967}
D.~G.~Aronson,
\emph{Bounds for the fundamental solution of a parabolic equation},
Bull.\ Amer.\ Math.\ Soc., 73 (1967), pp.~890--896.

\bibitem{bal-tal:1996}
V.~Bally and D.~Talay,
\emph{The law of the Euler scheme for stochastic differential equations: I.\
Convergence rate of the distribution function},
Probab.\ Theory Related Fields, 104 (1996), pp.~43--60.

\bibitem{bal-tal:1996b}
V.~Bally and D.~Talay,
\emph{The law of the Euler scheme for stochastic differential equations: II.\
Convergence rate of the density},
Monte Carlo Methods Appl., 2 (1996), pp.~93--128.

\bibitem{ben-etal:2023}
J.~Benton, V.~De~Bortoli, A.~Doucet, and G.~Deligiannidis,
\emph{Nearly $d$-linear convergence bounds for diffusion models via stochastic
localization},
in Proc.\ ICLR, 2024.

\bibitem{chen-etal:2023}
S.~Chen, S.~Chewi, J.~Li, Y.~Li, A.~Salim, and A.~Zhang,
\emph{Sampling is as easy as learning the score: theory for diffusion models with minimal data assumptions},
in Proc.\ ICLR, 2023.

\bibitem{con-etal:2025}
G.~Conforti, A.~Durmus, and M.~Gentiloni Silveri,
\emph{KL convergence guarantees for score diffusion models under minimal data
assumptions},
SIAM J.\ Math.\ Data Sci., 7 (2025), pp.~86--109.

\bibitem{friedman:1964}
A.~Friedman,
\emph{Partial Differential Equations of Parabolic Type},
Prentice-Hall, Englewood Cliffs, 1964.

\bibitem{gao-etal:2025}
X.~Gao, H.~M.~Nguyen, and L.~Zhu,
\emph{Wasserstein convergence guarantees for a general class of score-based
generative models},
J.\ Mach.\ Learn.\ Res., 26 (2025), no.~43, pp.~1--54.

\bibitem{hau-par:1986}
U.~G.~Haussmann and E.~Pardoux,
\emph{Time reversal of diffusions},
Ann.\ Probab., 14 (1986), pp.~1188--1205.

\bibitem{ho-etal:2020}
J.~Ho, A.~Jain, and P.~Abbeel,
\emph{Denoising diffusion probabilistic models},
in Proc.\ NeurIPS, 2020.

\bibitem{jain-zhang:2025}
N.~Jain and T.~Zhang,
\emph{A sharp KL-convergence analysis for diffusion models under minimal
assumptions},
in Proc.\ ICLR, 2026.

\bibitem{kom-bur:2014}
Y.~Komori and K.~Burrage,
\emph{A stochastic exponential Euler scheme for simulation of stiff
biochemical reaction systems},
BIT Numer. Math., 54 (2014), pp.~1067--1085.

\bibitem{lee-etal:2022}
H.~Lee, J.~Lu, and Y.~Tan,
\emph{Convergence for score-based generative modeling with polynomial complexity},
in Proc.\ NeurIPS, 2022.

\bibitem{lee-etal:2023}
H.~Lee, J.~Lu, and Y.~Tan,
\emph{Convergence of score-based generative modeling for general data distributions},
in Proc.\ ALT, 2023.

\bibitem{li-yan:2024}
G.~Li and Y.~Yan,
\emph{$O(d/T)$ convergence theory for diffusion probabilistic models under
minimal assumptions},
J. Mach. Learn. Res., 26 (2025), no.~292, pp.~1--55.

\bibitem{lip-shi:2001}
R.~S.~Liptser and A.~N.~Shiryaev,
\emph{Statistics of Random Processes I: General Theory},
Springer, 2nd ed., 2001.


\bibitem{nak:2024}
Y.~Nakano,
\emph{Structural and convergence analysis of discrete-time denoising diffusion
probabilistic models},
arXiv:2406.01320v7, 2026.

\bibitem{nakano:2026kernel}
Y.~Nakano,
\emph{A kernel-based method for Schr\"odinger bridges},
Japan J.\ Indust.\ Appl.\ Math., 43 (2026), Art.~46.



\bibitem{sheu:1991}
S.~J.~Sheu,
\emph{Some estimates of the transition density of a nondegenerate diffusion Markov process},
Ann.\ Probab., 19 (1991), pp.~538--561.

\bibitem{sohl-etal:2015}
J.~Sohl-Dickstein, E.~A.~Weiss, N.~Maheswaranathan, and S.~Ganguli,
\emph{Deep unsupervised learning using nonequilibrium thermodynamics},
in Proc.\ ICML, 2015.

\bibitem{son-etal:2020}
Y.~Song, J.~Sohl-Dickstein, D.~P.~Kingma, A.~Kumar, S.~Ermon, and B.~Poole,
\emph{Score-based generative modeling through stochastic differential equations},
in Proc.\ ICLR, 2021.

\bibitem{tsy:2009}
A.~B.~Tsybakov,
\emph{Introduction to Nonparametric Estimation},
Springer, New York, 2009.

\bibitem{vincent:2011}
P.~Vincent,
\emph{A connection between score matching and denoising autoencoders},
Neural Comput., 23 (2011), pp.~1661--1674.

\end{thebibliography}
\end{document}